\documentclass{whrartcl}

\usepackage[inference]{semantic}
\usepackage{mathpartir}
\usepackage{xcolor}
\usepackage{url}
\usepackage{bussproofs}
\usepackage{thm-restate}
\usepackage{varwidth}

\title{Unravelling Cyclic First-Order Arithmetic}
\author{Graham E.\ Leigh, University of Gothenburg\\ Dominik
  Wehr\thanks{Corresponding author; E-mail: dominik.wehr@gu.se}, University of Gothenburg}

\DeclareUnicodeCharacter{03BC}{$\mu$}
\usepackage[
  style=numeric-comp,
  sorting=nyt,
  maxbibnames=99,
  backref=false,
  doi=false,
  isbn=false,
  url=false,
  eprint=false,
  backend=biber
  ]{biblatex}
\bibliography{bibliography}

\newcommand{\RR}{\mathcal{R}}

\newcommand{\HH}{\mathcal{H}}

\newcommand{\HA}{\mathrm{HA}}
\newcommand{\mFOL}{\mu\mathrm{FOL}}
\newcommand{\PA}{\mathrm{PA}}
\newcommand{\CPA}{\mathrm{CPA}}
\newcommand{\CHA}{\mathrm{CHA}}
\newcommand{\AHA}{\mathrm{SHA}} 
\newcommand{\APA}{\mathrm{SPA}}

\newcommand{\ACA}{\mathrm{ACA}_0}

\newcommand{\lcls}{\Theta}
\newcommand{\pcls}[1]{\Pi(#1 \sdash #1)}
\newcommand{\fcls}{\pcls{\lcls}}
\newcommand{\IS}{\mathrm{I}\Sigma}
\newcommand{\CS}{\mathrm{C}\Sigma}
\newcommand{\CD}{\mathrm{C}\Delta}
\newcommand{\IP}{\mathrm{I}\Pi}
\newcommand{\CP}{\mathrm{C}\Pi}
\newcommand{\IC}{\mathrm{I}\lcls}
\newcommand{\CC}{\mathrm{C}\lcls}
\newcommand{\rcls}[1]{\pcls{#1} \cup \Delta_0}
\newcommand{\IPar}[1]{\mathrm{I}(\rcls{#1})}
\newcommand{\IF}{\IPar{\lcls}}

\newcommand{\ol}[1]{\overline{#1}}
\newcommand{\abs}[1]{|#1|}
\newcommand{\uh}{\upharpoonright}

\newcommand{\sdash}{\Rightarrow}

\newcommand{\FV}{\mathrm{FV}}

\newcommand{\Inf}{\mathrm{Inf}}

\newcommand{\form}{\textsc{Form}}
\newcommand{\term}{\textsc{Term}}
\newcommand{\var}{\textsc{Var}}
\newcommand{\Seq}{\textsc{Seq}}
\newcommand{\Leaf}{\mathrm{Leaf}}

\newcommand{\Inner}{\mathrm{Inner}}
\newcommand{\Chld}{\mathrm{Chld}}
\newcommand{\Nds}{\mathrm{Nodes}}
\newcommand{\im}{\mathrm{im}} 
\newcommand{\Down}{\mathrm{Down}}
\newcommand{\dom}{\mathrm{dom}}

\newcommand{\RAx}{\textsc{Ax}}
\newcommand{\RInd}{\textsc{Ind}}
\newcommand{\RWk}{\textsc{Wk}}

\newcommand{\RCut}{\textsc{Cut}}

\newcommand{\RCase}{\textsc{Case}}
\newcommand{\RDrop}{\textsc{Pop}}
\newcommand{\RComp}{\textsc{Comp}}
\newcommand{\RBud}{\textsc{Bud}}
\newcommand{\RBin}{\textsc{Bin}}
\newcommand{\RUn}{\textsc{Un}}

\newcommand*{\langeq}{\mathrel{\vcenter{\baselineskip0.5ex \lineskiplimit0pt
      \hbox{\scriptsize.}\hbox{\scriptsize.}}}%
  \mathrel{\vcenter{\baselineskip0.5ex \lineskiplimit0pt
      \hbox{\scriptsize.}\hbox{\scriptsize.}}}%
  =}

\colorlet{dwnote}{wpurple}
\colorlet{gelnote}{cyan}
\colorlet{insertcolour}{cyan}
\colorlet{deletecolour}{red}
\usetikzlibrary{arrows}

\begin{document}

\maketitle

\begin{abstract}
  Cyclic proof systems for Heyting and Peano arithmetic eschew induction axioms
  by accepting proofs which are finite graphs rather than trees. Proving that such
  a cyclic proof system coincides with its more conventional variants is often
  difficult: Previous proofs in the literature rely on intricate arithmetisations
  of the metamathematics of the cyclic proof systems.

  In this article, we present a simple and direct embedding of cyclic proofs for Heyting
  and Peano arithmetic into
  purely inductive, i.e. `finitary', proofs by adapting a translation introduced
  by Sprenger and Dam for a cyclic proof system of $\mFOL$ with explicit ordinal
  approximations. We extend their method to recover Das' result of $\CP_n \subseteq
\IP_{n + 1}$ for Peano arithmetic.
  As part of
  the embedding we present a novel representation of cyclic proofs as a labelled
  sequent calculus.
\end{abstract}


\textsc{Funding:} This work was supported by the Knut and Alice Wallenberg Foundation [2020.0199].

\section{Introduction}
\label{sec:intro}

Cyclic proof systems admit derivations whose underlying structures are finite
graphs, rather than finite trees.
Such proof systems are especially well-suited to logics and theories which
feature fixed points or (co-)inductively defined objects.
Compared to conventional `inductive' proof systems, cyclic proof systems often do not
require explicit induction rules, which are instead simulated by cycles. This
makes cyclic proof systems particularly well-suited to proof theoretic
investigations and applications in automated theorem proving, which are both
often hampered by induction rules.

There are two common methods for proving that a given cyclic proof system
corresponds to a particular logic or theory.
The usual method is via direct soundness and completeness arguments.
This approach is often not well-suited to arithmetical
theories.
In such cases, equivalence is instead established via proof theoretic translations
between
the cyclic and a conventional inductive proof system.
Translating inductive proofs into cyclic proofs
is usually straightforward whereas the opposite direction requires considerable effort.

For Peano and Heyting arithmetic, two general strategies for obtaining such
translations are present in the literature.
Simpson's approach~\cite{simpsonCyclicArithmeticEquivalent2017},
later refined by Das~\cite{dasLogicalComplexityCyclic2020},
formalises a soundness argument for cyclic proofs in subsystems of
second-order arithmetic, using reflection and conservativity results to
`return to' the first-order setting. Berardi and
Tatsuta~\cite{berardiEquivalenceInductiveDefinitions2017,berardiEquivalenceIntuitionisticInductive2017}
instead rely on Ramsey-style order-theoretic principles which can be
formalised in Heyting arithmetic. Both
approaches involve the formalisation of complex mathematical
concepts in first- and second-order arithmetic.

The literature contains other translation methods, although not applied to
arithmetical theories. Of significance to this article is a method first
employed by Sprenger and Dam~\cite{sprengerStructureInductiveReasoning2003a} for
first-order logic with fixed points $\mFOL$. This approach unfolds a cyclic
proof until inductive hypotheses can simply be `read off' the proof's structure.
The inductive proof resulting from this procedure thus naturally corresponds to
an unfolding of the initial cyclic proof.

In this article, we adapt Sprenger and Dam's `unfolding' method to present a proof translation between
the cyclic and inductive proof system for Heyting arithmetic.
The strategy can be adapted to Peano
arithmetic with minor adjustments.
Our construction of the inductive hypothesis in the translated proof refines
those employed by Sprenger and Dam, which allows recovery Das' logical
complexity result: In cyclic $\PA$, proofs with $\Pi_n$ cycles correspond to
$\IP_{n + 1}$-proofs. We obtain a similar result for $\HA$.
As part of the proof, we introduce a novel syntactic representation of the
cyclic proofs resulting from the unfoldings called
\emph{stack-controlled cyclic proofs}. These allow a clear description the
process of extracting invariants from such proofs
and thus lead to a simpler argument for
the translation from cyclic proofs to inductive proofs.
While we only present a stack-controlled system for Heyting arithmetic, the
notion readily generalises to cyclic proof systems for other logics.
Furthermore, we provide detailed summaries of previous approaches to translating
cyclic proof systems for Peano or Heyting arithmetic to inductive proof systems.

\paragraph{Overview}
\Cref{sec:cha} introduces the inductive and cyclic
proof systems for Heyting arithmetic.
\Cref{sec:sha} presents \emph{stack-labelled} proofs of cyclic Heyting
arithmetic, a proof system designed to ease the translation.
In \Cref{sec:translation}, we present the translation of cyclic, stack-labelled
proofs to plain, inductive proofs.
The full translation result is obtained in \Cref{sec:combinatorics} by proving that every proof
in cyclic Heyting arithmetic can be unfolded into a stack-labelled proof.
We conclude in \Cref{sec:conclusion} by extending the result to Peano
arithmetic, analysing the translation in terms of logical and proof size
complexity and discussing the wider applicability of the method presented in
this article.

\section{Cyclic Heyting Arithmetic}
\label{sec:cha}

The term and formula language of first-order arithmetic is given below, fixing a
countable set of variables $\var$. Write $\FV(\varphi)$
for the set of \emph{free variables} which occur unbound in a formula $\varphi$.
\begin{align*}
  s,t \in \term \langeq~& x ~|~ 0 ~|~ S s ~|~ s + t ~|~ s \cdot t  & x \in \var \\
  \varphi, \psi \in \form \langeq~& s = t ~|~ \bot ~|~ \varphi \wedge \psi ~|~ \varphi \vee \psi ~|~ \varphi \to \psi ~|~ \forall x.\varphi ~|~ \exists x. \varphi
\end{align*}
We define $\neg \varphi \coloneq \varphi \to \bot$ and $\top \coloneq \bot \to \bot$.
Denote by $[t / x]$ the usual \emph{substitution operation}, substituting the term
$t$ into all free occurrences of the variable $x$ in a term or formula. This is
a partial operation, $\varphi[t / x]$ being undefined if the free variables in
$t$ are not distinct from the bound variables in $\varphi$. Henceforth, writing
$\varphi[t / x]$ implicitly asserts that the substitution is defined.
We sometimes write $\varphi(x)$ as a shorthand for
$\varphi[x / y]$ for an unspecified variable $y \neq x$ if $x \not\in
\FV(\varphi)$. By analogy, the notation $\varphi(t)$ in the vicinity of
$\varphi(x)$ then denotes $\varphi[t / y]$.

We begin by recalling the inductive proof system for Heyting arithmetic.

\begin{definition}[Heyting arithmetic]
  The \emph{sequents} of \emph{Heyting arithmetic} are expressions $\Gamma \sdash \delta$ where
  $\Gamma$ is a finite set and $\delta$ a singular formula of first-order arithmetic.
  Write $\Gamma, \varphi$ for
  $\Gamma \cup \{\varphi\}$ and $\Gamma, \Gamma'$ for $\Gamma \cup \Gamma'$.
  The \emph{derivation rules} of $\HA$ comprise of the following choice of standard rules for
  intuitionistic first-order logic with equality,
  \begin{mathpar}
    \inference[\RAx]{}{\Gamma{}, \delta{} \sdash{} \delta}

    \inference[$\to$L]{\Gamma{}, \varphi{} \sdash{} \delta{} \quad \Gamma{} \sdash{}
      \psi{}}{\Gamma{}, \varphi{} \to \psi{} \sdash{} \delta}

    \inference[$\to$R]{\Gamma{}, \varphi{} \sdash{} \psi{}}{\Gamma{} \sdash{} \varphi{} \to \psi{}}

    \inference[$\wedge$L]{\Gamma{}, \varphi{}, \psi{} \sdash{} \delta{}}{\Gamma{},
      \varphi{} \wedge{} \psi{} \sdash \delta{}}

    \inference[$\wedge$R]{\Gamma \sdash \varphi \quad \Gamma \sdash \psi}{\Gamma \sdash \varphi \wedge
      \psi}

    \inference[$\vee$L]{\Gamma, \varphi \sdash \delta \quad \Gamma, \psi \sdash \delta}{\Gamma, \varphi \vee{}
      \psi \sdash \delta}

    \inference[$\vee$R]{\Gamma{} \sdash \varphi_i}{\Gamma{} \sdash
      \varphi_1 \vee{} \varphi_2}

    \inference[$\forall$L]{\Gamma, \varphi[t / x] \sdash \delta}{\Gamma, \forall x.\varphi \sdash \delta}

    \inference[$\forall$R]{\Gamma \sdash \varphi \quad x \not\in
      \FV(\Gamma, \forall x. \varphi)}{\Gamma
      \sdash \forall x.\varphi}

    \inference[$\exists$L]{\Gamma, \varphi \sdash \delta \quad x \not\in
      \FV(\Gamma, \exists x. \varphi, \delta)}{\Gamma, \exists x. \varphi \sdash \delta}

    \inference[$\exists$R]{\Gamma \sdash \varphi[t / x]}{\Gamma \sdash \exists x. \varphi}

    \inference[$\bot$L]{}{\Gamma, \bot \sdash \delta}

    \inference[$=$L]{\Gamma[x / y] \sdash \delta[x / y] \quad x,
      y \not\in \FV(s, t)}{\Gamma[s / x, t / y], s = t \sdash \delta[s / x, t
      / y]}

    \inference[$=$R]{}{\Gamma \sdash t = t}
  \end{mathpar}
  with the following two structural rules,
  \begin{mathpar}
    \inference[\RWk]{\Gamma \sdash \delta}{\Gamma, \Gamma' \sdash \delta}
  
    \inference[\RCut]{\Gamma \sdash \varphi \quad \Gamma, \varphi \sdash \delta}{\Gamma \sdash \delta}
  \end{mathpar}
  the following axiomatic sequents which characterise the function symbols of
  first-order arithmetic,
  \begin{align*}
    & \sdash 0 \neq S t & S s = S t & \sdash s = t && \sdash s + 0 = 0 \\
    &\sdash s + St = S(s + t) & & \sdash s \cdot 0 = 0 && \sdash s \cdot St = (s \cdot t) + s
  \end{align*}
  and the axiomatic sequents of \emph{induction}, for any formula $\varphi$,
  \[ \varphi(0), \forall x. \varphi(x) \to \varphi(Sx) \sdash \varphi(s) . \]

  Formally, a proof in Heyting arithmetic is
  a pair $\pi = (T, \rho)$ consisting of a finite tree $T$ and a function
  $\rho : \Nds(T) \mapsto \RR$, assigning instances of derivation rules to nodes of
  $T$. We define $\lambda : \Nds(T) \mapsto \Seq$ to map a $t \in \Nds(T)$ to
  the conclusion of $\rho(t)$. If $t \in \Nds(T)$ and $\rho(t)$ is a rule with $n$ premises $\Gamma_1 \sdash
  \delta_1, \ldots, \Gamma_n \sdash \delta_n$ then $\Chld(t)$ must be $t_1,
  \ldots, t_n$ such that $\lambda(t_i) = \Gamma_i \sdash
  \delta_i$. We consider axiomatic sequents as premiseless derivation rules.
  We call the sequent $\lambda(r)$ labelling the root of $T$ the \emph{endsequent}.
  If a proof with endsequent $\Gamma \sdash \delta$
  exists then $\Gamma \sdash \delta$ is \emph{provable in Heyting arithmetic},
  denoted by $\HA \vdash \Gamma \sdash \delta$.
\end{definition}

Recall that the order $x < y$ can be defined as $\exists z. y = x + Sz$ and $x
\leq y$ as $x < Sy$. Further, we employ the shorthands $\forall x < t. \varphi
\coloneq \forall x.~x < t \to \varphi$ and $\exists x < t. \varphi \coloneq
\exists x.~x < t \wedge \varphi$ for terms $t$.
The
following admissible rules and structural property are relied on by the
translation given in \Cref{sec:translation}.

\begin{lemma}\label{lem:rind}
  The following derivation rules are admissible in $\HA$
  for $x \not\in \FV(\Gamma, \varphi)$.
  \begin{mathpar}
    \inference[$\RInd$]{\Gamma, \forall y < x.~\varphi[y / x] \sdash \varphi}{\Gamma \sdash \forall x. \varphi}

    \inference[$\RCase_x$]{\Gamma(0) \sdash \delta(0) \quad \Gamma(Sx) \sdash
      \delta(Sx)}{\Gamma(x) \sdash \delta(x)}

    \inference[$\forall$L*]{\Gamma, \varphi[s / x] \sdash \delta}{\Gamma, s < t,
      \forall x < t.\varphi \sdash \delta}

    \inference[$\forall$R*]{\Gamma, x < t \sdash \varphi \quad x \not\in
      \FV(\Gamma, \forall x < t. \varphi)}{\Gamma \sdash \forall x < t.\varphi}
  \end{mathpar}
\end{lemma}

\begin{fact}\label{lem:ha-ren}
  If $\HA \vdash \Gamma \sdash \delta$ then $\HA \vdash \Gamma[t / x] \sdash
  \delta[t / x]$ for any term $t$ and variable $x$.
\end{fact}

We proceed by defining a cyclic proof system $\CHA$ for Heyting arithmetic. It eschews
the induction axioms in favour of \emph{cycles} which allow leaves to
be considered closed if they are labelled by a sequent that appears elsewhere in
the proof.

A \emph{cyclic tree} is a pair $C = (T, \beta)$ consisting of a finite tree $T$ and
a partial function $\beta \colon \Leaf(T) \to \Inner(T)$ mapping some leaves of
$T$ to its inner nodes.
If $t \in \dom(\beta)$ we call it a \emph{bud} and $\beta(t)$ \emph{its companion}.

\begin{definition}[$\CHA$-pre-proofs]\label{def:pre-proof}
  The sequents of \emph{cyclic Heyting arithmetic} are the same as those of Heyting
  arithmetic. The \emph{derivation rules} of $\CHA$ are the same as those of
  $\HA$ except the induction axioms are replaced by a $\RCase$-rule:
  \[
    \inference[$\RCase_x$]{\Gamma(0) \sdash \delta(0)
      \quad \Gamma(Sx) \sdash \delta(Sx)}{\Gamma(x) \sdash \delta(x)}
  \]

  A \emph{pre-proof} of $\CHA$ is a pair $\pi = (C, \rho)$
  consisting of a cyclic tree $C = (T, \beta)$
  and a function $\rho : \Nds(T) \setminus \dom(\beta) \to \RR$ labelling all non-bud
  nodes with derivation rules. We define $\lambda : \Nds(T) \to \Seq$ as follows
  \[
    \lambda(t) \coloneq
    \begin{cases}
      \text{conclusion of } \rho(t) & t \in \dom(\rho) \\
      \text{conclusion of } \rho(\beta(t)) & t \in \dom(\beta)
    \end{cases}.
  \]
  If $t \in \Nds(T)$ and $\rho(t)$ is a rule with $n$ premises $\Gamma_1 \sdash
  \delta_1, \ldots, \Gamma_n \sdash \delta_n$ then $\Chld(t)$ must be $t_1,
  \ldots, t_n$ such that $\lambda(t_i) = \Gamma_i \sdash \delta_i$.
\end{definition}

The structures defined by \Cref{def:pre-proof} are called pre-proofs, rather than
proofs, as they need not be sound. Such is the case for the pre-proof of
$0 \neq 1$ below, the two nodes marked with $\star$ forming a cycle.

\begin{comfproof}
  \AXC{$\sdash 0 + 0 = 0$}
  \AXC{$\sdash 0 \neq 1~\star$}
  \RightLabel{\small\RWk}
  \UIC{$0 + 0 = 0 \sdash 0 \neq 1$}
  \RightLabel{\small\RCut}
  \BIC{$\sdash 0 \neq 1~\star$}
\end{comfproof}

Cyclic Heyting arithmetic, like most cyclic proof systems, thus requires an
additional \emph{soundness condition} delineating between proofs and mere
pre-proofs. The soundness condition we give below is a so-called \emph{global
  trace condition}. In \Cref{sec:induction-orders}, we discuss induction orders
which are an alternative soundness condition for $\CHA$.

\begin{definition}[$\CHA$-proofs]\label{def:cha-proof}
  Let $\pi = (C, \rho)$ be a pre-proof.
  Consider an infinite branch $(\Gamma_i \sdash \delta_i)_{i \in \omega}$
  through $\pi$.
  A variable $x$ is said to have \emph{a
    trace along} $(\Gamma_i \sdash \delta_i)_{i \in \omega}$ if there exists an
   $n \in \omega$ such that $x \in \FV(\Gamma_i, \delta_i)$ for all $i \geq n$.
   Such a trace on $x$ is said to be \emph{progressing} if it passes through
   instances of the $\RCase_x$-rule (for the same variable \( x \)) infinitely often.

  A pre-proof $\pi$ is a \emph{proof} in cyclic Heyting arithmetic if for every
  infinite branch through $\pi$ there exists a variable which has a progressing trace along
  it. $\CHA \vdash \Gamma \sdash \delta$ denotes provability in $\CHA$.
\end{definition}

To illustrate the global trace condition, we prove that $\CHA$ is
sound over the standard model of the natural numbers $\omega$. 
For $V \subseteq \var$, an assignment is a function $\rho : V \to \omega$ mapping each 
first-order variable in $V$ to a natural number.
The satisfaction relation $\rho \vDash \Gamma \sdash \delta$, for an assignment
\( \rho \) with $\dom(\rho) \supseteq \FV(\Gamma, \delta)$ and sequent $\Gamma \sdash
\delta$, is
defined in the usual manner.
In the following, \( \omega \vDash \Gamma \sdash \delta \) expresses that for
all suitable assignments \( \rho \), we have \( \rho \vDash \Gamma \sdash \delta \).

\begin{proposition}[Soundness]\label{lem:sound}
  If $\CHA \vdash \Gamma \sdash \delta$ then $\omega \vDash \Gamma \sdash \delta$.
\end{proposition}
\begin{proof}
  Suppose $\pi$ was a $\CHA$-proof of $\Gamma \sdash \delta$ and suppose, towards
  contradiction, that there was an assignment $\rho : \FV(\Gamma,\delta) \to \omega$ such that
  $\rho \not\vDash \Gamma \sdash \delta$. Consider the $\CHA$-rule applied
  to the root of $\pi$, deriving $\Gamma \sdash \delta$: As the rule is sound, it must have
  a premise $\Gamma_1 \sdash \delta_1$ for which there exists an assignment
  $\rho_1$ with $\rho_1 \not\vDash \Gamma_1 \sdash \delta_1$. Note that
  $\rho_1$ can always be chosen to extend $\rho$, except for the variable $x$ if the last applied rule
  is $\RCase_x$ and $\rho(x) > 0$, in which case $\rho_1(x) = \rho(x) - 1$. This argument can
  be iterated infinitely, following the
  cycles `back down' the proof when reaching buds, because the
  contradicted premise of a derivation rule can never be an axiom.
  Thus, iterating this
  argument yields an infinite branch $(\Gamma_i \sdash \delta_i)_{i \in \omega}$
  through $\pi$ and a sequence $(\rho_i)_{i \in \omega}$ of contradicting
  assignments. Then there exists a progressing $x$-trace along $(\Gamma_i \sdash
  \delta_i)_{i \in \omega}$ and by the aforementioned relation of the $(\rho_i)_{i \in
    \omega}$, the sequence $(\rho_i(x))_{i \in \omega}$ in $\omega$ never
  increases but decreases whenever $\RCase_x$ is passed, which takes place
  infinitely often. Such an infinitely decreasing sequence contradicts the
  well-foundedness of $\omega$.
\end{proof}

The system $\CHA$ is a variation of the cyclic proof system for Heyting
arithmetic considered by Berardi and Tatsuta
in~\cite{berardiEquivalenceIntuitionisticInductive2017}. Their system is an
extension of Brotherston's cyclic proof system for first-order logic with
inductive definitions~\cite{brotherstonSequentCalculusProof2006} with the usual
axiomatisations of $s + t, s \cdot t$ and $S t$. They employ a
unary inductive predicate $N(t)$ to assert that the term $t$ denotes a natural
number, their $\RCase$-rule and trace condition being defined in terms of the
predicate. As in our system every term denotes a natural number, we have chosen
to not include the predicate $N t$ and adjust the $\RCase$-rule accordingly.
A crucial difference to Berardi and Tatsuta's system, which we rely on
throughout the article, is that traces only ever follow a single variable,
without branching or rejoining of traces. This makes the notion of `progress' in
our system much simpler, easing the combinatorics of \Cref{sec:combinatorics}.

$\HA$ and $\CHA$ prove the same theorems. The direction from $\HA$ to $\CHA$
follows from the fact that induction can be simulated by the use of cycles in
$\CHA$, as is shown below. The opposite direction is much more difficult to
prove and the subject of the remaining sections of this article.

\begin{theorem}
  If $\HA \vdash \Gamma \sdash \delta$ then $\CHA \vdash \Gamma \sdash \delta$.
\end{theorem}
\begin{proof}
  We begin by showing that $\CHA$ proves every instance $\varphi(0), \forall x.
  \varphi(x) \to \varphi(Sx) \sdash \varphi(s)$ of the induction scheme.
  Take $\ol{\varphi}$ to be a shorthand for the left-hand side of the induction scheme. The two
  sequents marked with a $\star$ form a cycle.
  \begin{comfproof}
    \AXC{}
    \LSC{\textsc{Ax}}
    \UIC{$\ol{\varphi} \sdash \varphi(0)$}
    \AXC{$\ol{\varphi} \sdash \varphi(x) ~~\star$}
    \AXC{}
    \RSC{\textsc{Ax}}
    \UIC{$\ol{\varphi}, \varphi(S\,x) \sdash \varphi(S\,x)$}
    \RSC{$\forall$L,$\to$L}
    \BIC{$\ol{\varphi} \sdash \varphi(S\,x)$}
    \LSC{$\textsc{Case}_x$}
    \BIC{$\ol{\varphi} \sdash \varphi(x) ~~\star$}
    \LSC{$\forall$R}
    \UIC{$\ol{\varphi} \sdash \forall x. \varphi(x)$}
    \AXC{}
    \RSC{\textsc{Ax}}
    \UIC{$\ol{\varphi}, \varphi(s) \sdash \varphi(s)$}
    \RSC{$\forall$L}
    \UIC{$\ol{\varphi}, \forall x. \varphi(x) \sdash \varphi(s)$}
    \insertBetweenHyps{\hskip -6ex}
    \LSC{\textsc{cut}}
    \BIC{$\ol{\varphi} \sdash \varphi(s)$}
  \end{comfproof}
  There exists precisely one infinite branch through the pre-proof above, following the
  $\star$-cycle infinitely. This infinite branch has a progressing $x$-trace starting in
  the premise of $\forall$R. Thus, the pre-proof above constitutes a proof.

  Towards the original claim, suppose $\HA \vdash \Gamma \sdash \delta$. A
  $\HA$-proof $\pi$ witnessing this fact differs from a $\CHA$-pre-proof only by the
  fact that some of its leaves may be labelled with instances of the induction
  schemes. Thus, a $\CHA$-pre-proof $\pi'$ of $\Gamma \sdash \delta$ may be obtained by
  `grafting' copies of the $\CHA$-proof given above onto each such leaf of
  $\pi$. As there are no cycles in $\pi$, every infinite branch of $\pi'$
  must follow the $\star$-cycle of one of the `grafted on' induction scheme
  proofs and thus has a successful trace. Hence $\pi'$ is a $\CHA$-proof
  witnessing $\CHA \vdash \Gamma \sdash \delta$.
\end{proof}

\section{Stack-labelled Heyting arithmetic}
\label{sec:sha}

Roughly, the translation of $\CHA$-proofs into $\HA$-proofs can be split into
two steps: First $\CHA$-proofs are arranged into a normal form, codified by an intermediate proof system
$\AHA$. As the second step, $\AHA$-proofs are translated into $\HA$ proofs.
All `logical content' of the translation is limited to the second step, the
first step being pure combinatorics. We thus begin by presenting the second
step of the translation, deferring the combinatorial arguments to
\Cref{sec:combinatorics}. This section introduces the system $\AHA$ of
\emph{stack-labelled Heyting arithmetic}. In \Cref{sec:translation}, we present
the translation of $\AHA$-proofs to $\HA$-proofs.

The sequents of $\AHA$ are expressions $\Lambda \mid \Gamma \sdash \delta$,
consisting of an $\HA$ sequent $\Gamma \sdash \delta$ and an ordered list
$\Lambda$ of \emph{companion labels}, which is treated as a stack. We denote the
empty list by $\varepsilon$. Intuitively,
the companion labels in $\Lambda$ collect the sequents of \emph{companions} to
which one may `cycle back' at the current point of the proof.
Formally, a \emph{companion label} is an expression of the shape $x^\bullet \mapsto \Gamma
\sdash \delta$ for a $\HA$-sequent $\Gamma \sdash \delta$, a variable $x \in
\FV(\Gamma, \delta)$ and where $x^\bullet \in \{x^+, x^-\}$. We write
$\var(\Lambda)$ for the set $\{x \mid (x^\bullet \mapsto \Gamma \sdash \delta)
\in \Lambda\}$ of variables with associated companion labels in $\Lambda$.
Returning to the intuitive reading of companion labels, $x^+
\mapsto \Gamma \sdash \delta$ in $\Lambda$ indicates that the current point of
the proof may `cycle back' to a companion node with $\HA$-sequent $\Gamma \sdash
\delta$ because the companion's associated variable $x$ has been subject to a
$\RCase$-rule between the companion and the current point in the proof.
Conversely, a companion label $x^- \mapsto \Gamma \sdash \delta$ indicates that
$x$ has not yet been subject to $\RCase$ and thus no cycle may be formed yet.

A $\AHA$-sequent $\Lambda \mid \Gamma \sdash \delta$ is \emph{well-formed} if
for every $x^\bullet \mapsto \Delta \sdash \gamma \in \Lambda$ we have $x \in
\FV(\Gamma, \delta)$. Going forward, we assume that every $\AHA$-sequent
$\Lambda \mid \Gamma \sdash \delta$ under consideration is well-formed, unless
stated otherwise.

Before presenting the rules of the calculus \( \AHA \), we fix further
operations on stacks of companion labels. Given a stack $\Lambda$ and $n \in
\omega$, we define $\Lambda \uh n$ to be the longest prefix $\Lambda'$ of
$\Lambda$ with $\abs{\Lambda'} \leq n$. Given a
variable \( x \), we define $\Lambda^{+x}$ to be the result of replacing in
$\Lambda$ all occurrences of $x^-$ by $x^+$:
\[
  \varepsilon^{+x} \coloneq \varepsilon
  \qquad
  (y^\bullet \mapsto \Gamma \sdash \delta; \Lambda)^{+x} \coloneq
  \begin{cases}
     x^- \mapsto \Gamma \sdash \delta; \Lambda^{+x}, &\text{if \( y = x \),}
    \\
    y^\bullet \mapsto \Gamma \sdash \delta; \Lambda^{+x} , &\text{otherwise.}
  \end{cases}
\]

\begin{definition}
  The \emph{derivation rules} of $\AHA$ are of split into three categories:
  \begin{description}
	\item [Logical rules,] which do not interact with companion labels:
  \begin{mathpar}
    \inference[\RAx]{}{\varepsilon \mid \Gamma, \delta \sdash \delta}

    \inference[$\to$L]{\Lambda \mid \Gamma{}, \varphi{} \sdash{} \delta{} \quad \Lambda \mid \Gamma{} \sdash{}
      \psi{}}{\Lambda \mid \Gamma{}, \varphi{} \to \psi{} \sdash{} \delta}

    \inference[$\to$R]{\Lambda \mid \Gamma{}, \varphi{} \sdash{} \psi{}}{\Lambda \mid \Gamma{} \sdash{} \varphi{} \to \psi{}}
    
    \\

    \inference[$\wedge$L]{\Lambda \mid \Gamma{}, \varphi{}, \psi{} \sdash{} \delta{}}{\Lambda \mid \Gamma{},
      \varphi{} \wedge{} \psi{} \sdash \delta{}}

    \inference[$\wedge$R]{\Lambda \mid \Gamma \sdash \varphi \quad \Lambda \mid \Gamma \sdash \psi}{\Lambda \mid \Gamma \sdash \varphi \wedge
      \psi}

    \inference[$\vee$L]{\Lambda \mid \Gamma, \varphi \sdash \delta \quad \Lambda \mid \Gamma, \psi \sdash \delta}{\Lambda \mid \Gamma, \varphi \vee{}
      \psi \sdash \delta}

    \inference[$\vee$R]{\Lambda \mid \Gamma{} \sdash \varphi_i}{\Lambda \mid \Gamma{} \sdash
      \varphi_1 \vee{} \varphi_2}
    
    \\

    \inference[$\forall$L]{\Lambda \mid \Gamma, \varphi[t / x] \sdash \delta}{\Lambda \mid \Gamma, \forall x.\varphi \sdash \delta}

    \inference[$\forall$R]{\Lambda \mid \Gamma \sdash \varphi \quad x \not\in
      \FV(\Lambda \mid \Gamma, \forall x. \varphi)}{\Lambda \mid \Gamma
      \sdash \forall x.\varphi}

    \inference[$\exists$L]{\Lambda \mid \Gamma, \varphi \sdash \delta \quad x \not\in
      \FV(\Lambda \mid \Gamma, \exists x. \varphi, \delta)}{\Lambda \mid \Gamma, \exists x. \varphi \sdash \delta}

    \inference[$\exists$R]{\Lambda \mid \Gamma \sdash \varphi[t / x]}{\Lambda \mid \Gamma \sdash \exists x. \varphi}

    \inference[$=$L]{\Lambda \mid \Gamma[x / y] \sdash \delta[x / y] \quad x, y \not\in \FV(s, t)}{\Lambda \mid \Gamma[s / x, t / y], s = t \sdash \delta[s / x, t / y]}

    \inference[$=$R]{}{\Lambda \mid \Gamma \sdash t = t}
    
	\inference[$\bot$L]{}{\Lambda \mid \Gamma, \bot \sdash \delta}

  \inference[\RWk]{\Lambda \mid \Gamma \sdash \delta}{\Lambda \mid \Gamma, \Gamma' \sdash \delta}

  \inference[\RCut]{\Lambda \mid \Gamma \sdash \varphi \quad \Lambda \mid \Gamma, \varphi \sdash \delta}{\Lambda \mid \Gamma \sdash \delta}
  \end{mathpar}
  \item [Arithmetical axioms] of the form:
  \begin{align*}
    \varepsilon \mid \Gamma &\sdash 0 \neq S t & \varepsilon \mid \Gamma , S s = S t & \sdash s = t & \varepsilon \mid \Gamma &\sdash s + 0 = 0 \\
    \varepsilon \mid \Gamma &\sdash s + St = S(s + t) & \varepsilon \mid \Gamma &\sdash s \cdot 0 = 0 &\varepsilon \mid \Gamma &\sdash s \cdot St = (s \cdot t) + s
  \end{align*}
  \item [Label rules] which manipulate companion labels:
  \begin{mathpar}
    \inference[$\RComp$]{\Lambda; (x^- \mapsto \Gamma \sdash \delta) \mid \Gamma \sdash \delta}{\Lambda \mid \Gamma \sdash \delta} 

    \inference[$\RBud$]{}{\Lambda; (x^+ \mapsto \Gamma \sdash \delta) \mid \Gamma \sdash \delta}

    \inference[$\RDrop$]{\Lambda \mid \Gamma \sdash \delta}{\Lambda; \Lambda' \mid \Gamma \sdash \delta}

    \inference[$\RCase_{x}$]{\Lambda \mid \Gamma(0) \sdash \delta(0)
      \qquad \Lambda^{+x} \mid \Gamma(S\,x) \sdash \delta(S\,x)}{\Lambda \mid \Gamma(x) \sdash \delta(x)}
  \end{mathpar}
\end{description}

  We only consider finite, non-cyclic derivations in $\AHA$ to be proofs. That is,
  a proof in $\AHA$ is
  a pair $\pi = (T, \rho)$ consisting of a finite tree $T$ and a function
  $\rho : \Nds(T) \mapsto \RR$, assigning instances of $\AHA$-derivation rules to nodes of
  $T$. If $t \in \Nds(T)$ and $\rho(t)$ is a rule with $n$ premises $\Lambda_1
  \mid \Gamma_1 \sdash
  \delta_1, \ldots, \Lambda_n \mid \Gamma_n \sdash \delta_n$ then $\Chld(t)$ must be $t_1,
  \ldots, t_n$ such that the conclusion of the $\rho(t_i)$ is $\Lambda_i \mid \Gamma_i \sdash
  \delta_i$. We consider axiomatic sequents as premiseless derivation rules.
  Every rule but $\RDrop$
  may only applied to derive sequents well-formed sequents.
  We call the conclusion of the rule labelling the root of $T$ the \emph{endsequent}.
  If such a
  proof with endsequent $\varepsilon \mid \Gamma \sdash \delta$
  exists then $\Gamma \sdash \delta$ is \emph{provable in stack-controlled Heyting arithmetic},
  denoted $\AHA \vdash \Gamma \sdash \delta$.
\end{definition}

Formally, $\AHA$ is not a cyclic proof system because its derivations are
non-cyclic trees. However, through the $\RComp$- and $\RBud$-rules, it retains a
`cyclic character'. In this sense, it can be considered as a mid-point between
the fully cyclic system $\CHA$ and the fully non-cyclic system $\HA$.

The exception stipulating that $\RDrop$ may derive ill-formed sequents is in
place to cover for a situation which may arise in $\RCase$-applications.
Suppose $\Lambda = \Lambda_0 ; (x^\bullet \mapsto \Delta \sdash
\gamma); \Lambda_1$ where $x \not\in \var(\Lambda_0)$.
Consider the following application of
$\RCase_x$:
\[
    \inference[$\RCase_{x}$]{\Lambda \mid \Gamma(0) \sdash \delta(0)
      \qquad \Lambda^{+x} \mid \Gamma(S\,x) \sdash \delta(S\,x)}{\Lambda \mid \Gamma(x) \sdash \delta(x)}
\]
Then $x \not\in \FV(\Gamma(0), \delta(0))$ in the left-hand premise, but
$\Lambda$ still contains the label $x^\bullet \mapsto \Delta \sdash \gamma$,
making the left-hand premise an ill-formed $\AHA$-sequent. This may be rectified
by an application of the $\RDrop$-rule, reducing $\Lambda$ to $\Lambda_0$,
yielding a well-formed sequent, at which point other derivation rules may be
applied again.

The following key theorem is proven in \Cref{sec:combinatorics}:

\begin{restatable}{theorem}{chatoaha}\label{lem:cha-to-aha}
  If $\CHA \vdash \Gamma \sdash \delta$ then $\AHA \vdash \Gamma \sdash \delta$.
\end{restatable}

\section{Translating from $\AHA$ to $\HA$}
\label{sec:translation}

This section describes a translation of $\AHA$-proofs into $\HA$-proofs.
The general strategy is to transform
each $\AHA$-derivation $\pi$ of $\Lambda \mid \Gamma \sdash \delta$ into a
$\HA$-derivation of $\HH_\Lambda(\pi), \Gamma \sdash \delta$ where
$\HH_\Lambda(\pi)$ is a formula `computed from the data of $\Lambda$ and $\pi$'
called the \emph{induction invariant}. While this general strategy is similar to
the approach of Sprenger and
Dam~\cite{sprengerStructureInductiveReasoning2003a}, some of the details
differ: Their method would yield a sequent $\HH_\Lambda, \Gamma \sdash \delta$
where $\HH_\Lambda$ is a set of inductive hypotheses `computed from the data of
$\Lambda$'. By also taking into account the shape of $\pi$ in computing
$\HH_\Lambda(\pi)$, we are able to combine the inductive hypotheses into a
single formula, which allows us to control the inductive complexity of the
resulting $\HA$-proof, mirroring a result obtained by
Das~\cite{dasLogicalComplexityCyclic2020} for cyclic $\PA$.

The construction of the inductive hypothesis $\HH_\Lambda(\pi)$ can be split
into multiple steps. First, we define $I_\Lambda$, which should be considered
the `inductive hypothesis of the companion label on the top of the stack $\Lambda$'.
Thus, let $\Lambda$ be non-empty, i.e.\ $\Lambda \coloneq \Lambda'; x^\bullet \mapsto \Gamma \sdash \delta$.
We separate the free variables of $\Gamma \sdash \delta$, other than $x$, into
$\vec{v} = \var(\Lambda) \setminus \{x\}$ and
$\vec{w} = \FV(\Gamma, \delta) \setminus (\var(\Lambda) \cup \{x\})$. Then
define
\[
  I_\Lambda \coloneq
  \begin{cases}
    \forall x' \,<\, x.~\widehat{I}_\Lambda [x' / x] & \text{ if } \bullet = - \\
    \forall x' \,\leq\, x.~\widehat{I}_\Lambda [x' / x] & \text{ if } \bullet = +
  \end{cases} \quad \text{ where } \quad
  \widehat{I}_\Lambda \coloneq \forall \vec{u} \leq \vec{v}. \forall \vec{w}.~(\bigwedge
  \Gamma \to \delta)[\vec{u} / \vec{v}].
\]
\begin{proposition}\label{lem:end-lem}
  \
  \begin{enumerate}[(i)]
  \item $\FV(I_\Lambda) = \FV(\widehat{I}_\Lambda) = \var(\Lambda)$
  \item $\HA \vdash I_\Lambda, x \leq y \sdash I_\Lambda[x / y]$ for any
    $\Lambda$ and $x,y \in \var$
  \item $\HA \vdash I_\Lambda[Sx / x] \sdash I_{\Lambda^{+x}}$ for $x \in \FV(I_\Lambda)$
  \end{enumerate}
\end{proposition}

The inductive hypothesis $\HH_\Lambda(\pi)$ is defined in terms of the recursive
function $H_n(\pi)$ which calculates the inductive hypothesis for a derivation
$\pi \vdash \Lambda \mid \Gamma \sdash \Delta$ where $n \leq \abs{\Lambda}$.
Intuitively, $H_n(\pi)$ computes `the inductive hypothesis for the first $n$
labels of $\Lambda$ along $\pi$'. We thus fix $\HH_\Lambda(\pi) \coloneq H_{\abs{\Lambda}}(\pi)$.
\newsavebox{\budpt}
\begin{lrbox}{\budpt}
  \begin{varwidth}{\linewidth}
    \begin{comfproof}
      \AXC{}
      \LSC{$\RBud$}
      \UIC{$\Lambda; (x^+ \mapsto \Gamma \sdash \delta) \mid \Gamma \sdash \delta$}
    \end{comfproof}
  \end{varwidth}
\end{lrbox}
\newsavebox{\droppt}
\begin{lrbox}{\droppt}
  \begin{varwidth}{\linewidth}
    \begin{comfproof}
      \AXC{$\pi'$}
      \LSC{$\RDrop$}
      \UIC{$\Lambda; \Lambda' \mid \Gamma \sdash \delta$}
    \end{comfproof}
  \end{varwidth}
\end{lrbox}
\newsavebox{\casept}
\begin{lrbox}{\casept}
  \begin{varwidth}{\linewidth}
    \begin{comfproof}
      \AXC{$\pi_0$}
      \AXC{$\pi_s$}
      \LSC{$\RCase_x$}
      \BIC{$\Lambda \mid \Gamma \sdash \delta$} 
    \end{comfproof}
  \end{varwidth}
\end{lrbox}
\newsavebox{\unpt}
\begin{lrbox}{\unpt}
  \begin{varwidth}{\linewidth}
    \begin{comfproof}
      \AXC{$\pi'$}
      \LSC{\textsc{Un}}
      \UIC{$\Lambda \mid \Gamma \sdash \delta$}
    \end{comfproof}
  \end{varwidth}
\end{lrbox}
\newsavebox{\binpt}
\begin{lrbox}{\binpt}
  \begin{varwidth}{\linewidth}
    \begin{comfproof}
      \AXC{$\pi_l$}
      \AXC{$\pi_r$}
      \LSC{\textsc{Bin}}
      \BIC{$\Lambda \mid \Gamma \sdash \delta$}
    \end{comfproof}
  \end{varwidth}
\end{lrbox}
\begin{align*}
  H_0(\pi) & \coloneq \top \\
  H_n\left(
    \hspace{-0.5em}\usebox{\budpt}
  \right) & 
      \coloneq
                                                           \begin{cases}
                                                             I_{\Lambda}
                                                             & \text{ if } \abs{\Lambda} = n \\
                                                             \top & \text{ otherwise}
                                                           \end{cases} \\
  H_n \left( 
    \hspace{-0.6em}\usebox{\droppt}
  \right) & 
   \coloneq
                                      H_{\min(\abs{\Lambda}, n)}(\pi)
                                    \\
  H_{n} \left(
    \hspace{-0.6em}\usebox{\unpt}
  \right) & \coloneq H_{n}(\pi) \\
  H_{n}\left(
    \hspace{-0.6em}\usebox{\binpt}
  \right)& \coloneq H_{n}(\pi_l) \wedge H_{n}(\pi_r) \\
  H_n \left(
    \hspace{-0.6em}\usebox{\casept}
  \right) &
                                 \coloneq \begin{cases}
                                   \begin{aligned}
    \forall x' & \leq x.~(x' = 0 \wedge H_{n}(\pi_0)[x' / x]) \vee \\ & (\exists x''.~x' = Sx'' \wedge H_{n}(\pi_s)[x'' / x]) 
                                   \end{aligned}
                                   & \text{ if } x \in \var(\Lambda \uh n) \\
                                   H_{n}(\pi_0) \wedge H_{n}(\pi_s) & \text{ otherwise}
                                 \end{cases}
\end{align*}
The clauses $\textsc{Un}$ and $\textsc{Bin}$ are catch-all for unary and binary
rules not covered by the other clauses.

Intuitively, for a proof $\pi \vdash \Lambda \mid \Gamma \sdash \delta$, the
partial inductive hypothesis $H_n(\pi)$ accounts for all instances of the invariants $x
\mapsto \Gamma' \sdash \delta' \in \Lambda \uh n$ needed to `close the proof'. They are collected by
recursively traversing the proof tree $\pi$. For example in the $\RBin$-case, the instances needed
to close the whole proof are precisely those needed to close the left and the right
premise, hence their conjunction is chosen.

The most complicated clause of $H_n(\pi)$ is that of the $\RCase$-rule. Observe
that it is an embodiment of the theorem $\forall y.~y = 0 \vee \exists
y'.~y = Sy'$ of arithmetic. The variable $x$ in the conclusion of $\pi_S$ thus
corresponds to $y'$ in the theorem while that same $x$ corresponds to $y$ in the
conclusion of $\RCase$. The complicated case of the $\RCase$-clause, if $x \in
\var(\Lambda \uh n)$ and thus possibly $x \in \FV(H_n(\pi))$, accounts for this
change of reference. However, this alone does not explain the structure of
$H_n(\pi)$ in this case. There are two more considerations which
influence it: First, it is important that the sequent $H_n(\pi), x \leq y \sdash
H_n(\pi)[x / y]$ is provable in $\HA$ (\Cref{lem:ih-lem} (ii) below), for which
the `downwards closing' with the prefix of $\forall x' \leq x.$ is required.
Second, we have designed the $\RCase$-clause such that the resulting $H_n(\pi)$
embodies the aforementioned case-distinction theorem. In the translation of a
$\AHA$-proof to a $\HA$-proof, this will allow us to translate every application of
a $\RCase$-rule to a variable in $\var(\Lambda)$ as an `interaction' with
$\HH_\Lambda(\pi)$ in a straight-forward manner.

\begin{lemma}\label{lem:ih-lem}
  \
  \begin{enumerate}[(i)]
  \item $\FV(H_n(\pi \vdash \Lambda \mid \Gamma \sdash \delta)) \subseteq \var(\Lambda \uh n)$
  \item $\HA \vdash H_n(\pi), x \leq y \sdash H_n(\pi)[x / y]$ for any $\pi$,
    $n$ and $x, y \in \var$.
  \item Let $\abs{\Lambda} > n$ then $\HA \vdash H_{n}(\pi), I_{\Lambda \uh n +
      1} \sdash H_{n + 1}(\pi)$.
  \end{enumerate}
\end{lemma}
\begin{proof}
  \
  \begin{enumerate}[(i)]
  \item Follows immediately from \Cref{lem:end-lem} (i) and by scrutinising the
    definition of $H_n(-)$.
  \item Proceed by induction on $\pi$. Only 3 cases are of interest because they
    influence the free variables of $H_\bullet(\pi)$:
    \begin{description}
    \item[Bud:] Then $H_n(\pi) = \top$ or $H_n(\pi) = I_{\Lambda}$ for some
      $\Lambda$. In the latter case, the claim follows directly
      from \Cref{lem:end-lem} (ii).
    \item[Pop:] It might be the case that $y \not\in H_n(\pi)$ because the
      corresponding label is erased by $\RDrop$. In this case, simply close by
      $\RAx$. Otherwise continue via the IH.
    \item[$\RCase_z$:] The only interesting case is if $z = y$. In this case,
      \[
        H_n(\pi) = \forall y' \leq y.~\underbrace{(y' = 0 \wedge H_n(\pi_0)[y' / y]) \vee
        (\exists y''.~y' = sy'' \wedge H_n(\pi_s)[y''/ y])}_{\varphi }.
      \]
      Then simply derive the following:
      \begin{comfproof}
        \AXC{}
        \LSC{$\RAx$}
        \UIC{$A, y' \leq y \sdash A$}
        \LSC{$\forall$L*}
        \UIC{$\forall y' \leq y. A, y' \leq y \sdash A$}
        \UIC{$\forall y' \leq y. A, x \leq y, y' \leq x \sdash A$}
        \LSC{$\forall$R*}
        \UIC{$\forall y' \leq y. A, x \leq y \sdash \forall y' \leq x. A$}
      \end{comfproof}
    \end{description}
  \item 
  Per induction on $\pi$.
  \begin{description}
  \item[Bud:] If $\abs{\Lambda} \neq n + 1$, this is trivial as then
    $H_{n + 1}(\pi) = \top$. If $\abs{\Lambda} = n + 1$ then
    $H_{n + 1} = I_{\Lambda \uh n + 1}$ which follows immediately.
  \item[Pop:] The only interesting case is if $\Lambda$ is partitioned into $\Lambda_1 ; \Lambda_2$ with $\Lambda_2$
    being discarded by the rule and $\abs{\Lambda_1} < n + 1$. In
    this case, $H_{n}(\pi) = H_{\abs{\Lambda_1}}(\pi') =
    H_{n + 1}(\pi)$ and the proof can be closed with $\RAx$.
  \item[$\RCase_y$:] The case in which $y \not\in \var(\Lambda \uh n + 1)$ can be
    treated as the generic binary case. Thus assume $y \in \var(\Lambda \uh n + 1)$. Then
    $H_{n + 1}(\pi) = \forall y' \leq y.~L \vee R$ with $L =
    y' = 0 \wedge H_{n + 1}(\pi_0)[y' / y]$ and $R = \exists y''.~y' = sy''
    \wedge H_{n + 1}(\pi_S)[y'' / y]$. Write $I \coloneq I_{\Lambda \uh n + 1}$
    and $I^+ \coloneq I_{(\Lambda \uh n + 1)^{+y}}$.
    We must distinguish two cases:
    \begin{description}
    \item[$y \in \var(\Lambda \uh n)$:] Let \( \psi_0 \coloneq y' = 0 \wedge H_{n}(\pi_0)[y' / y] \) and \( \psi_s \coloneq \exists y''.~y' = sy'' \wedge H_{n}(\pi_s)[y'' / y] \). We derive:
      \begin{scprooftree}{1}
        \AXC{$\pi_L$}
        \UIC{$\psi_L, I[y' / y] \sdash L \vee R$}

        \AXC{$\pi_R$}
        \UIC{$\psi_s, I[y' / y] \sdash L \vee R$}

        \LSC{$\vee$L}
        \BIC{$\psi_0 \vee \psi_s , I[y' / y] \sdash L \vee R$}
        \LSC{\Cref{lem:end-lem} (ii), $\RWk$}
        \UIC{$\psi_0 \vee \psi_s , I,  y' \leq y \sdash L \vee R$}
        \LSC{$\forall$L*}
        \UIC{$\forall y' \leq y.~\psi_0 \vee \psi_s , I, y' \leq y \sdash L \vee R$}
        \LSC{$\forall$R*}
        \UIC{$\forall y' \leq y.~\psi_0 \vee \psi_s , I \sdash \forall y' \leq y. L \vee R$}
      \end{scprooftree}
      \vspace{0.5em}
      where \( \pi_L \) and \( \pi_R \) are, respectively, the proofs
      \begin{scprooftree}{1}
        \AXC{$y' = 0 \sdash y' = 0 $}
        \AXC{IH of $\pi_0$}
        \UIC{$H_{n}(\pi_0), I \sdash H_{n + 1}(\pi_0)$}
        \LSC{\Cref{lem:ha-ren}}
        \UIC{$H_{n}(\pi_0)[y' / y], I[y' / y] \sdash H_{n + 1}(\pi_0)[y' / y]$}
        \LSC{$\wedge$R, $\wedge$L, $\RWk$}
        \BIC{$\psi_0 , I[y' / y] \sdash y' = 0 \wedge H_{n + 1}(\pi_0)[y' / y]$}
        \LSC{$\vee$R}
        \UIC{$\psi_0 , I[y' / y] \sdash L \vee R$}
      \end{scprooftree}
      \begin{scprooftree}{1}
        
        \AXC{$y' = sy'' \sdash y' = sy''$}
        \AXC{IH of $\pi_s$}
        \UIC{$H_{n}(\pi_s), I^+ \sdash H_{n + 1}(\pi_s)$}
        \RSC{\Cref{lem:ha-ren}}
        \UIC{$H_{n}(\pi_s)[y'' / y], I^+[y'' / y] \sdash H_{n + 1}(\pi_s)[y'' / y] $}
        \RSC{\Cref{lem:end-lem} (iii)}
        \UIC{$H_{n}(\pi_s)[y'' / y], I[sy'' / y] \sdash H_{n +
            1}(\pi_s)[y'' / y] $}
        \RSC{$\exists$L, $\vee$R, $\exists$R, $\wedge$L, $\wedge$R}
        \BIC{$\psi_s, I[y' / y] \sdash L \vee R$}
      \end{scprooftree}
      \vspace{.5em}
    \item[$y \in \var(\Lambda \uh n + 1) \setminus \var(\Lambda \uh n)$:] Observe that the proof must
      have the form:
      \begin{scprooftree}{0.88}
        \AXC{$\pi'_0$}
        \UIC{$\Lambda' \mid \Gamma[0 / y] \sdash \delta[0 / y]$}
        \RSC{$\pi_0$}
        \LSC{$\RDrop$}
        \UIC{$\Lambda_1; y^\bullet \mapsto \Gamma' \sdash \delta'; \Lambda_2 \mid
          \Gamma[0 / y] \sdash \delta[0 / y]$}
        \AXC{$\pi_s$}
        \UIC{$\Lambda_1; y^+ \mapsto \Gamma' \sdash \delta'; \Lambda_2 \mid
          \Gamma[sy / y] \sdash \delta[sy / y]$}
        \LSC{$\RCase_y$}
        \BIC{$\Lambda_1; y^\bullet \mapsto \Gamma' \sdash \delta'; \Lambda_2 \mid
          \Gamma \sdash \delta$}
      \end{scprooftree}
      \vspace{.5em}
      with $\abs{\Lambda_1} = n$ and $\abs{\Lambda'} \leq n$. Hence,
      $H_{n}(\pi_0) = H_{\abs{\Lambda'}}(\pi_0') = H_{n + 1}(\pi_0)$ with $y
      \not\in \FV(H_{\abs{\Lambda'}}(\pi_0'))$, which we shall refer to by $(\star)$.
      Further, observe that $y \not\in \FV(H_{n}(\pi_S))$.
      Using the result of the induction hypothesis applied to \( \pi'_0 \) and \( \pi_s \), we derive:
      \begin{scprooftree}{1}
        \AXC{$\sdash
          0 = 0$}
        \AXC{IH of $\pi'_0$}
        \UIC{$H_{\abs{\Lambda'}}(\pi_0) \sdash H_{\abs{\Lambda'}}(\pi_0)$}
        \LSC{$\wedge$R}
        \BIC{$H_{\abs{\Lambda'}}(\pi_0) \sdash 0 = 0 \wedge H_{\abs{\Lambda'}}(\pi_0)$}
        \LSC{$(\star)$}
        \UIC{$H_{n}(\pi_0)[0 / y] \sdash 0 = 0 \wedge H_{n + 1}(\pi_0)[0 / y]$}

        \AXC{$\pi_0$}
        \LSC{$\RCase_{y'}, \RWk, \vee$R}
        \BIC{$H_{n}(\pi_0), H_{n}(\pi_s)[y' / y], I[y' / y] \sdash L
          \vee R$}
        \LSC{$y \not\in \FV(H_n(\pi_S))$}
        \UIC{$H_{n}(\pi_0), H_{n}(\pi_s), I[y' / y], y' \leq y \sdash L
          \vee R$}
        \LSC{\Cref{lem:end-lem} (ii)}
        \UIC{$H_{n}(\pi_0), H_{n}(\pi_s), I, y' \leq y \sdash L
          \vee R$}
        \LSC{$\wedge$L, $\forall$R*}
        \UIC{$H_{n}(\pi_0) \wedge H_{n}(\pi_s), I \sdash \forall y' \leq y. L
          \vee R$}
      \end{scprooftree}
      where \( \pi_0 \) is the derivation 
      \begin{scprooftree}{1}
        \AXC{IH of $\pi_s$}
        \UIC{$H_{n}(\pi_s), I \sdash H_{n + 1}(\pi_s)$}
        \RSC{\Cref{lem:ha-ren}}
        \UIC{$H_{n}(\pi_s)[y'' / y], I[y'' / y]
          \sdash H_{n + 1}(\pi_s)[y'' / y]$}
        \RSC{(ii)}
        \UIC{$H_{n}(\pi_s)[sy'' / y], I^+[y'' / y]
          \sdash H_{n + 1}(\pi_s)[y'' / y]$}
        \RSC{\Cref{lem:end-lem} (iii)}
        \UIC{$H_{n}(\pi_s)[sy'' / y], I[sy'' / y]
          \sdash H_{n + 1}(\pi_s)[y'' / y]$}
        \RSC{$\exists$R, $\wedge$R, $=$R}
        \UIC{$H_{n}(\pi_s)[sy'' / y], I[sy'' / y]
          \sdash \exists z.~sy'' = sz \wedge H_{n + 1}(\pi_s)[z / y]$}
      \end{scprooftree}
      \vspace{.5em}
    \item[Un:] The inductive hypothesis directly yields the desired claim.
    \item[Bin:] This is easily handled via $\wedge$R and subsequent $\wedge$L,
      weakenings and application of the corresponding IH.
    \end{description}
  \end{description}
\end{enumerate}
\end{proof}

\begin{theorem}\label{lem:key}
  If $\AHA \vdash \Gamma \sdash \delta$ then $\HA \vdash \Gamma \sdash \delta$.
\end{theorem}
\begin{proof}
  Let $\pi$ be a the derivation of $\varepsilon \mid \Gamma \sdash \delta$ in $\AHA$.
  We will prove a generalisation of the claim: If $\pi$ is
  an $\AHA$-derivation of $\Lambda \mid \Gamma \sdash \delta$ then $\HA \vdash
  \HH_\Lambda(\pi), \Gamma \sdash \delta$. For the original claim, this yields
  $\HA \vdash \HH_\varepsilon, \Gamma \sdash \delta$. Observe that
  $\HH_\varepsilon(\pi) = H_0(\pi)$ is always equivalent to $\top$, allowing us
  to conclude $\HA \vdash \Gamma \sdash \delta$ as desired. We now prove the
  generalisation by induction on the derivation $\pi$:
  \begin{description}
  \item[Bud:] In that case, $\HH_\Lambda(\pi) = I_\Lambda$, from which $\Gamma
    \to \delta$ readily follows.
  \item[$\RCase_x$:] Suppose the proof $\pi$ was of shape
    \begin{comfproof}
      \AXC{$\pi_0$}
      \UIC{$\Lambda' \mid \Gamma[0 / x] \sdash \delta[0 / x]$}
      \LSC{$\RDrop$}
      \UIC{$\Lambda \mid \Gamma[0 / x] \sdash \delta[0 / x]$}
      \AXC{$\pi_s$}
      \UIC{$\Lambda^{+x} \mid \Gamma[sx / x] \sdash \delta[sx / x]$}
      \LSC{$\RCase_x$}
      \BIC{$\Lambda \mid \Gamma \sdash \delta$}
    \end{comfproof}
    then derive
    \begin{comfproof}
      \AXC{IH for $\pi_0$}
      \UIC{$\HH_{\Lambda'}(\pi_0), \Gamma[0 / x] \sdash \delta[0 / x]$}
      \LSC{$\wedge$L, $=$L}
      \UIC{$x = 0 \wedge \HH_{\Lambda'}(\pi_0), \Gamma \sdash \delta$}
      \AXC{IH for $\pi_s$}
      \UIC{$\HH_\Lambda(\pi_s), \Gamma[sx / x] \sdash \delta[sx / x]$}
      \RSC{\Cref{lem:ha-ren}}
      \UIC{$\HH_\Lambda(\pi_s)[x'' / x], \Gamma[sx'' / x] \sdash \delta[sx'' / x]$}
      \RSC{$\exists$L, $\wedge$L, $=$L}
      \UIC{$\exists x'. x = sx' \wedge
        \HH_\Lambda(\pi_s)[x' / x], \Gamma \sdash \delta$}
      \LSC{$\vee$L}
      \BIC{$(x = 0 \wedge \HH_{\Lambda'}(\pi_0)) \vee (\exists x''. x = sx'' \wedge
        \HH_\Lambda(\pi_s)[x'' / x]), \Gamma \sdash \delta$}
      \LSC{$\forall$L*}
      \UIC{$\forall x' \leq x.~(x' = 0 \wedge \HH_{\Lambda'}(\pi_0)[x' / x]) \vee (\exists x''. x' = sx'' \wedge
        \HH_\Lambda(\pi_s)[x'' / x]), \Gamma \sdash \delta$}
    \end{comfproof}
    where $\star$ is an application of the usual renaming lemma.
  \item[Comp:] Suppose $\pi$ was of shape
    \begin{comfproof}
      \AXC{$\pi'$}
      \UIC{$\Lambda; x^- \mapsto \Gamma \sdash \delta \mid \Gamma \sdash \delta$}
      \UIC{$\Lambda \mid \Gamma \sdash \delta$}
    \end{comfproof}
    and fix $\Lambda' \coloneq \Lambda; x^- \mapsto \Gamma \sdash \delta$.
    Derive the following in $\HA$:
    \begin{scprooftree}{1}
      \AXC{\Cref{lem:ih-lem} (iii)}
      \UIC{$\HH_\Lambda(\pi), I_{\Lambda'} \sdash
        \HH_{\Lambda'}(\pi')$}
      \LSC{$\dagger$}

      \AXC{IH}
      \UIC{$\HH_{\Lambda'}(\pi'), \Gamma \sdash \delta$}
      \RSC{\Cref{lem:ha-ren}}
      \UIC{$\HH_{\Lambda'}(\pi')[\vec{u} / \vec{v}], \Gamma[\vec{u} / \vec{v}] \sdash \delta[\vec{u} / \vec{v}]$}
      \RSC{\Cref{lem:ih-lem} (ii)}
      \UIC{$\HH_{\Lambda'}(\pi'), \Gamma[\vec{u} / \vec{v}], \vec{u} \leq
        \vec{v} \sdash \delta[\vec{u} / \vec{v}]$}
      \RSC{$\forall$R*, $\to$L}
      \UIC{$\HH_{\Lambda'}(\pi') \sdash \widehat{I}_{\Lambda'}$}
      \LSC{$\RCut, \RWk$}
      \BIC{$\HH_\Lambda(\pi), I_{\Lambda'} \sdash \widehat{I}_{\Lambda'}$}
      \LSC{$\RInd$}
      \UIC{$\HH_\Lambda(\pi) \sdash \forall x. \widehat{I}_{\Lambda'}$}

      \AXC{$\Gamma, \bigwedge \Gamma \to \delta \sdash \delta$}
      \RSC{$\forall$L*}
      \UIC{$\Gamma, \forall x. \widehat{I}_{\Lambda'} \sdash \delta$}
      \LSC{$\RCut$}
      \insertBetweenHyps{\hskip -6em}
      \BIC{$\HH_\Lambda(\pi), \Gamma \sdash \delta$}
    \end{scprooftree}
    where the right-most leaf is a trivial proof and at $\dagger$ we use that $\HH_{\Lambda}(\pi) = \HH_{\Lambda}(\pi')$.
  \item[Pop:] Suppose $\pi$ was of shape
    \begin{comfproof}
      \AXC{$\pi'$}
      \UIC{$\Lambda_0 \mid \Gamma \sdash \delta$}
      \LSC{$\RDrop$}
      \UIC{$\Lambda_0; \Lambda_1 \mid \Gamma \sdash \delta$}
    \end{comfproof}
    Then $\HH_\Lambda(\pi) = \HH_{\Lambda_0; \Lambda_1}(\pi) = \HH_{\Lambda_0}(\pi')$
    and thus $\HA \vdash \HH_\Lambda(\pi), \Gamma \sdash \delta$ per IH of $\pi'$ directly.
  \item[Bin:] As $\RCase$ was already treated above, we may assume that $\RBin$
    directly corresponds to a logical rule $\RBin'$ of $\HA$. Thus $\pi$ is of the shape
    below, notably $\Lambda$ appearing unchanged in both premises.
    \begin{comfproof}
      \AXC{$\pi_l$}
      \noLine
      \UIC{$\Lambda \mid \Gamma_l \sdash \delta_l$}
      \AXC{$\pi_r$}
      \noLine
      \UIC{$\Lambda \mid \Gamma_r \sdash \delta_r$}
      \LSC{\textsc{Bin}}
      \BIC{$\Lambda \mid \Gamma \sdash \delta$}
    \end{comfproof}
    Then derive:
    \begin{comfproof}
      \AXC{IH of $\pi_l$}
      \UIC{$\HH_\Lambda(\pi_l), \Gamma_l \sdash \delta_l$}
      \LSC{$\wedge$L, $\RWk$}
      \UIC{$\HH_\Lambda(\pi_l) \wedge \HH_\Lambda(\pi_r), \Gamma_l \sdash \delta_l$}
      \RSC{$\wedge$L, $\RWk$}
      \AXC{IH of $\pi_r$}
      \UIC{$\HH_\Lambda(\pi_r), \Gamma_r \sdash \delta_r$}
      \UIC{$\HH_\Lambda(\pi_l) \wedge \HH_\Lambda(\pi_r), \Gamma_r \sdash \delta_r$}
      \LSC{$\RBin'$}
      \BIC{$\HH_\Lambda(\pi_l) \wedge \HH_\Lambda(\pi_r), \Gamma \sdash \delta$}
    \end{comfproof}
  \item[Un:] As $\RComp$ was already treated above, we may assume that $\RUn$
    directly corresponds to a logical or structural rule $\RUn'$ of $\HA$ and that
    the derivation is of the following shape, with $\HH_\Lambda(\pi) = \HH_{\Lambda'}(\pi')$:
    \begin{comfproof}
      \AXC{$\pi'$}
      \UIC{$\Lambda' \mid \Gamma' \sdash \delta'$}
      \LSC{$\RUn$}
      \UIC{$\Lambda \mid \Gamma \sdash \delta$}
    \end{comfproof}
    Then derive:
    \begin{comfproof}
      \AXC{IH}
      \UIC{$\HH_\Lambda'(\pi'), \Gamma' \sdash \delta'$}
      \LSC{$\RUn'$}
      \UIC{$\HH_\Lambda'(\pi'), \Gamma \sdash \delta$}
    \end{comfproof}
  \end{description}
\end{proof}

\section{Unfolding Cyclic Arithmetic}
\label{sec:combinatorics}

This section covers the translation of $\CHA$-proofs into $\AHA$-proofs. The
translation consists of two parts. First, $\CHA$-proofs are transformed into
$\CHA$-proofs of the same end-sequent in a normal-form which simplifies the
`interaction between cycles' within a proof. The $\CHA$-proofs in this normal-form
are in fact essentially unfoldings of the original $\CHA$-proofs, meaning computation
content is preserved. This first step is presented in
\Cref{sec:induction-orders} and directly follows Sprenger and
Dam~\cite{sprengerStructureInductiveReasoning2003a}, who show the analogous
result for $\mFOL$, first-order logic extended with least and greatest fixed
points. Their proof operates purely at the level of cyclic proofs and their
traces, not interacting with the logical properties of $\mFOL$ at all, and thus
readily transfers to $\CHA$ and $\CPA$. In the second step, $\CHA$-proofs in
said normal-form are translated $\AHA$-proofs, which we present in \Cref{sec:io-to-sha}.

We begin by fixing a concrete notion of \emph{tree} to work with in this section.
For $s, t \in \omega^*$, we write $s < t$ if $s$
is a strict prefix of $t$, and \( s \le t \) if either \( s < t \) or \( s = t \).
A \emph{tree} $T \subseteq \omega^*$ is a non-empty set of sequences
which is closed under prefixes: if $s \in T$ and $t < s$ then $t \in T$.
We denote by
{$[s, t]$} the path from $s$ to $t$, i.e. $[s, t] \coloneq \{u
\in \omega^* \mid s \leq u \leq t\}$.
We denote the extension of $s \in \omega^*$ with $n \in \omega$ by $s \cdot n$. Concatenation of sequences is denoted by juxtaposition such that \( s \le st \) for all \( s , t \in \omega^* \).
Each $t \in T$ is called a \emph{node} and the nodes
in $\Chld_T(t) \coloneq \{t \cdot i \in T ~|~ i \in \omega\}$ are called the
\emph{children} of \( t \). We omit the parameter $T$, writing $\Chld(t)$,
whenever this does not cause any confusion.
A node $t$ is a \emph{leaf} of $T$ if $\Chld(t) = \emptyset$
and an \emph{inner node} otherwise. We write \( \Leaf(T) \) and \( \Inner(T) \) for the set of leaves and inner nodes of \( T \), respectively.

A \emph{cyclic tree} is a pair $C = (T, \beta)$ of a finite tree $T$ and
a partial function $\beta \colon \Leaf(T) \to \Inner(T)$ mapping some leaves of $T$
onto inner nodes of $T$.
If $t \in \dom(\beta)$ one calls it a \emph{bud} and $\beta(t)$ its \emph{companion}.
For a bud $t \in \dom(\beta)$, we call the path $[\beta(t), t]$ the \emph{local
  cycle} of $t$ and denote it by $\gamma_C(t)$. We omit the parameter $C$,
writing $\gamma(c)$, whenever this causes no confusion.

Cyclic trees $(T, \beta)$ which satisfy $\beta(t) < t$ for every $t \in \dom(\beta)$ are
said to be in \emph{cycle normal-form}. Every cyclic tree can be `unfolded'
into cycle normal-form, although at an super-exponential size
increase~\cite[Theorem 6.3.6]{brotherstonSequentCalculusProof2006}. We will thus
assume, without loss of generality, that all cyclic trees in this section are in
cycle normal-form. Note that cycle normal-form is not the normal-form of
$\CHA$-proofs with which \Cref{sec:induction-orders} is concerned.

Fix a cyclic tree $C = (T, \beta)$. A \emph{finite path} $p$ through $C$ is a
non-empty sequence $p \in T^+$ such that for every $i < \abs{p} - 1$ we have
$p_{i + 1} \in \Chld(p_i)$ if $p_i \in \Inner(T)$ and $p_i \in
\dom(\beta)$ with $p_{i + 1} = \beta(p_i)$ otherwise. We say $p$
\emph{starts at} $p_0$ and \emph{ends at} $p_{\abs{p} - 1}$. An \emph{infinite branch} through $C$
is an infinite sequence $t \in T^\omega$ with $t_0 = \varepsilon$ and such
that for each $i \in \omega$ either $t_{i + 1} \in \Chld(t_i)$ or $t_i \in
\dom(\beta)$ and $t_{i + 1} = \beta(t_i)$. If $C$ is the underlying cyclic tree
of a $\CHA$-proof $(C, \rho)$, this induces an infinite sequence
$(\Gamma_i \sdash \delta_i)_{i \in \omega}$ of sequents with $\lambda(t_i) =
\Gamma_i \sdash \delta_i $, which we use interchangeably with $(t_i)_{i \in
  \omega}$ to denote an infinite branch.

\subsection{Induction Orders}
\label{sec:induction-orders}

Sprenger and Dam~\cite{sprengerStructureInductiveReasoning2003a} work with a
cyclic proof system using induction orders, a soundness condition of cyclic proofs different from the
global trace condition given in \Cref{def:cha-proof}. We begin by transferring this soundness
condition to the setting of Heyting arithmetic.

Induction orders are characterised in terms of the strongly connected subgraphs of cyclic
proofs, which can, in turn, be identified with a subset $\eta \subseteq
\dom(\beta)$.
\begin{definition}
  Fix a cyclic tree $C = (T, \beta)$ and let $\eta \subseteq \dom(\beta)$.
  \begin{enumerate}
  \item The set of nodes $C[\eta] \coloneq \bigcup_{s \in \eta} \gamma(s)$ is
    said to be part of $C$ \emph{covered} by the cycles corresponding to the buds in $\eta$.
  \item
    The set $\eta$ is \emph{strongly connected} if
    there a finite path $p$ through $C$ which (1) starts and ends at the same
    node and (2) visits precisely the nodes in $C[\eta]$, i.e.\ $p$ visits a node $t \in T$ if and only
    if $t \in C[\eta]$.
  \item A strongly connected $\eta$ is said to be \emph{$B$-maximal} for $B \subseteq \dom(\beta)$ if $\eta \subseteq B$ and there
      is no $\eta \subsetneq \eta' \subseteq B$ which is strongly connected.
  \item
    Let $p$ be an infinite path through $C$ then $\Inf(p) \coloneq \{s \in T
~|~ p_i = s \text{ infinitely often}\}$.
  \end{enumerate}
\end{definition}

The strongly connected components of a cyclic tree are closely linked to their infinite
branches: For any infinite branch through a cyclic tree, the nodes visited
infinitely often by it are precisely $C[\eta]$ for a strongly connected $\eta$.
A proof of \Cref{lem:inf-cs}, for an equivalent definition of
strongly connected component, result can be found in \cite[Lemma
5.3]{leighGTCResetGenerating2023}.

\begin{lemma}\label{lem:inf-cs}
  Let $t \in T^\omega$ be an infinite branch of a cyclic tree $C = (T, \beta)$.
  Then there exists a strongly connected component $\eta$ of $C$ such that $\text{Inf}(t) =
  C[\eta]$.
\end{lemma}

A key property of
$B$-maximal strongly connected components is that they
are disjoint iff they are distinct.

\begin{lemma}
  For a cyclic tree $C = (T, \beta)$ let $B \subseteq \dom(\beta)$ and let
  $\eta, \eta' \subseteq B$ be $B$-maximal strongly connected components. Then either $\eta = \eta'$ or $\eta
  \cap \eta' = \emptyset$.
\end{lemma}
\begin{proof}
  Suppose there was $s \in \eta \cap \eta'$. Then, w.l.o.g. there exists paths
  $p$ and $p'$ starting and ending at $s$ such that $p$
  and $p'$ visit precisely the nodes of $C[\eta]$ and $C[\eta']$,
  respectively. `Pasting together' $p$ and $p'$ yields a path from $s$ to
  $s$ visiting precisely the elements of $C[\eta] \cup C[\eta'] = C[\eta \cup
  \eta']$, meaning $\eta \cup \eta'$ is strongly connected. By the
  $B$-maximality of $\eta$ and $\eta'$ it then follows that $\eta = \eta \cup
  \eta' = \eta'$.
\end{proof}

Having established some properties of strongly connected components, we now define induction orders.

\begin{definition}[Induction order]\label{def:io}
  Let $\pi = (C, \rho)$ be a $\CHA$-pre-proof.
 An \emph{induction order} is a preorder
  $\mathrel{\preceq}\; \subseteq
  \dom(\beta) \times \dom (\beta) $ together with a mapping $x_{-} : \dom(\beta)
  \to \var$ such that $x_s \in \FV(\lambda(s))$ for all $s \in \dom(\beta)$
  and
  \begin{itemize}
  \item if $s \preceq t$ then $\gamma(s)$ \emph{preserves} $x_t$, i.e.\ $x_t \in
    \FV(\lambda(u))$ for every $u \in \gamma(s)$,
  \item the cycle $\gamma(s)$ \emph{progresses} $x_s$, i.e.\ a $\RCase_{x_s}$
    instance is applied along the path $\gamma(s)$,
  \item every strongly connected $\eta \subseteq \dom(\beta)$ has a
    $\preceq$-maximum, i.e.\ there is $s \in \eta$ such that $t \preceq s$ for
    all $t \in \eta$
  \end{itemize}
\end{definition}

In our cyclic system for Heyting arithmetic, every proof has an induction order.
This fact is closely linked to the notion of trace and the trace condition of
$\CHA$. Other cyclic proof systems for first-order arithmetic,
such as those considered by Simpson~\cite{simpsonCyclicArithmeticEquivalent2017}
or Berardi and Tatsuta~\cite{berardiEquivalenceIntuitionisticInductive2017}, do
not exhibit this property.
The proof of this property is based on \cite[Theorem
5.14]{sprengerGlobalInductionMechanisms2003} which proves the same property for
the $\mFOL$-system Sprenger and Dam consider.

\begin{restatable}{lemma}{existsio}
  \label{lem:exists-io}
  A $\CHA$-pre-proof $\pi$ is a proof iff there exists an induction order for it.
\end{restatable}
\begin{proof}
  For the backwards direction, consider a branch $t \in T^\omega$ through $\pi$.
  Then $\Inf(t) = C[\eta]$ for some strongly connected $\eta$ by \Cref{lem:inf-cs}.
  Notably, there must be some $n \in \omega$ such that $t_i \in C[\eta]$ for all
  $i > n$. As $\preceq$ is an induction order, $\eta$ must have a {$\preceq$-maximum $s
  \in \eta$}. Then $x_s \in \bigcap_{i > n}
  \var(\Gamma_i, \delta_i)$ because {$\gamma(u)$} for each $u \in \eta$ preserves $x_s$, meaning there is an $x_s$-trace along $t$. This trace
  is progressing as there is a $\RCase_{x_s}$-instance along $\gamma(s)$ which
  $t$ passes infinitely often.

  Conversely, suppose $\pi$ was a proof. We obtain an induction order on
  $\pi$ through an iterative process. This process produces a sequence $S_0,
  S_1, \ldots$ of sets of disjoint strongly connected components of $\pi$ such that at each
  step, some bud $s \in \dom(\beta)$ is removed from $\bigcup S_i$. To begin, define $S_0 \coloneq
  \{\eta \subseteq \dom(\beta) ~|~ \eta \text{ is a } \dom(\beta)\text{-maximal
    strongly connected component}\}$. To obtain $S_{i + 1}$, select some $\eta_i \in S_i$ and
  consider a branch $t \in T^\omega$ through $\pi$ such that $\Inf(t) = \eta_i$.
  As $\pi$ is a proof, there must be an $n \in \omega$ such that for some $x \in
  \bigcap_{n < j} \var(\Gamma_i, \delta_i)$ there exists a progressing $x$-trace
  along $p$. Hence, $x$ must be preserved along the cycle {$\gamma(t)$} for every $t
  \in \eta$ and there must be an $s \in \eta$ for which the cycle $\gamma(s)$
  progresses $x$. Fix $x_s \coloneq x$ and set
  \[S_{i + 1} \coloneq (S_{i} \setminus \eta_i) \cup \{\eta' \subseteq (\eta_i
    \setminus \{s\}) ~|~ \eta' \text{ is a } (\eta_i
    \setminus \{s\})\text{-maximal strongly connected component} \}.\]
  That is, $S_{i + 1}$ is obtained by replacing $\eta_i$ in $S_i$ by the maximal
  strongly connected {components partitioning} $\eta_i$ after
  removing $s$. {Because each $S_i$ consists of pairwise
    disjoint strongly connected components, each $S_i$ partitions $\bigcup
  S_i$.} This process terminates at $S_{\abs{\dom(\beta)}} = \emptyset$ as every
  step removes a bud from $\bigcup S_i$. It remains to define the preorder
  $\preceq$ on $\dom(\beta)$.
  For $s, t \in \dom(\beta)$, let $i$ be the least index 
  {
  such that $t \not\in \bigcup S_{i + 1}$, i.e.\ $t$ was removed at step $i$ of
  the procedure. We set $s \preceq t$ iff $s \in \eta_i$.
  Now consider a strongly connected $\eta \subseteq \dom(\beta)$.} We claim that
  there is a $\preceq$-maximum of $\eta$. As the
  elements of $S_0$ are $\dom(\beta)$-maximal, there must be $\eta' \in S_0$
  such that $\eta \subseteq \eta'$. Now at each step of the process, {if $\eta'
  \in S_i$} then either
  $\eta' = \eta_i$ or $\eta' \in S_{i + 1}$. Because $S_{\abs{\dom(\beta)}} =
  \emptyset$ there thus must be $\eta_i = \eta'$. Consider the bud $s$ which is
  removed from $\eta_i$ to construct $S_{i + 1}$; if $s \not\in \eta$ then there
  is a $(\eta_i \setminus \{s\})$-maximal $\eta'' \in S_{i + 1}$ such that $\eta
  \subseteq \eta''$. Again, because $S_{\abs{\dom(\beta)}} = \emptyset$ there eventually
  must be a step $i$ such that $\eta \subseteq \eta_i$ and the bud $s$ removed
  from $\eta_i$ is in $\eta$. Then $s$ is a $\preceq$-maximum in $\eta_i$, and thus
  in $\eta$, per definition of $\preceq$.
  Reflexivity of $\preceq$ is trivial. Transitivity of $\preceq$ follows from an
  argument, similar to that for $\preceq$-maxima, along the
  iterative process, observing that for $s \preceq
  t \preceq u$ with corresponding indexes $t \not\in \bigcup S_{i + 1}, u \not\in
  \bigcup S_{j + 1}$ and $s \in \eta_i, t \in \eta_j$, $\eta_i \subsetneq \eta_j$ because each step `separates'
  some $\eta_k$ into strongly connected sub-components $\eta' \subsetneq \eta_k$.
\end{proof}

The following combinatorial lemma is crucial to the proof normalisation of $\CHA$-proofs.
A pre-proof $\pi = ((T, \beta), \lambda)$ is called \emph{injective} if the
function $\beta$ is injective.

\begin{lemma}\label{lem:exists-injective}
  If $\CHA \vdash \Gamma \sdash \delta$ then there exists an injective $\CHA$-proof of
  $\Gamma \sdash \delta$.
\end{lemma}
\begin{proof}
  Suppose $\pi = ((T, \beta), \rho)$ was a proof of $\Gamma \sdash \delta$.
  To obtain an injective proof of $\Gamma \sdash \delta$, apply the following
  procedure iteratively: For $u \neq v \in \dom(\beta)$ with $\beta(u) =
  \beta(v) = t$ and $\lambda(t) = \Gamma' \sdash \delta$ as depicted on the
  left-hand side of \Cref{f-injective}, add an application of the $\RWk$-rule which does not
  remove any formulas from $\Gamma'$ directly above $t$, `generating' a new node
  $t'$ above $t$ with $\lambda(t') = \lambda(t) = \Gamma' \sdash \delta$ and
  `redirect' $\beta(u) \coloneq t'$, as depicted on the right-hand side. After finitely
  many iteration steps, this process yields an injective pre-proof. As these transformations
  do not affect the traces through the
  pre-proof, it remains a proof.
\end{proof}
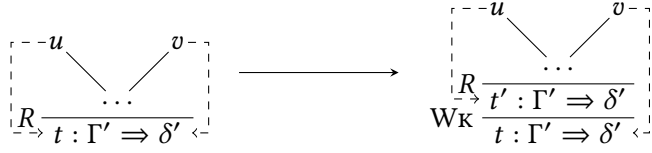
\begin{figure}
\centering
\begin{tikzpicture}
  \node[inner sep=1] (u) at (-0.8, 1.2) {$u$};
  \node[inner sep=1] (v) at (0.8, 1.2) {$v$};
  \node (sp) at (0, 0.4) {$\ldots$};
  \draw (sp) -- (v);
  \draw (sp) -- (u);
  \draw (-1, 0.2) -- (1, 0.2);
  \node at (-1.2, 0.2) {$R$};
  \node (t) at (0, 0) {$t : \Gamma' \sdash \delta'$};

  \draw[->, dashed] (u) -- (-1.4, 1.2) -- (-1.4, 0) -- (t);
  \draw[->, dashed] (v) -- (1.2, 1.2) -- (1.2, 0) -- (t);

  \draw[-{stealth}] (1.6, 0.8) -- (3.7, 0.8);

  \node[inner sep=1] (u) at (5, 1.65) {$u$};
  \node[inner sep=1] (v) at (6.6, 1.65) {$v$};
  \node (sp) at (5.8, 0.85) {$\ldots$};
  \draw (sp) -- (v);
  \draw (sp) -- (u);
  \draw (4.8, 0.65) -- (6.8, 0.65);
  \node at (4.6, 0.65) {$R$};
  \node (t') at (5.8, 0.45) {$t' : \Gamma' \sdash \delta'$};
  \draw (4.8, 0.2) -- (6.8, 0.2);
  \node at (4.4, 0.2) {$\RWk$};
  \node (t) at (5.8, 0) {$t : \Gamma' \sdash \delta'$};

  \draw[->, dashed] (u) -- (4.4, 1.65) -- (4.4, 0.45) -- (t');
  \draw[->, dashed] (v) -- (7, 1.65) -- (7, 0) -- (t);
\end{tikzpicture}
\caption{Transformation into injective cyclic proofs}
\label{f-injective}
\end{figure}

The translation procedure of Sprenger and
Dam~\cite{sprengerStructureInductiveReasoning2003a} operates on cyclic proofs
whose induction order coincides with the structure of the proof's cycles.
The \emph{structural dependency
  order $\sqsubseteq$} is defined on $\dom(\beta)$ such that $s \sqsubseteq t$
holds iff $\beta(s) \in \gamma(t)$, i.e. if the companion of $s$ occurs along
the local cycle of $t$. The transitive closure of \( \sqsubseteq \) is denoted \( \sqsubseteq^* \).
Sprenger and Dam show that every $\mFOL$-proof can be unfolded into a proof
with induction order $\sqsubseteq^*$. This result readily transfers to the
setting of $\CHA$.

\begin{theorem}
  \label{thm:ind-order-SD}
  For every injective \( \CHA \)-proof \( \pi \) with induction order \( \sqsubseteq \) there is an injective \( \CHA \)-proof \( \pi' \) with induction order \( \sqsubseteq^* \) employing only sequents in \( \pi \).
\end{theorem}
\begin{proof}
  Theorem 5 and Lemma 5 of \cite{sprengerStructureInductiveReasoning2003a}.
  While the proof system considered 
  in~\cite{sprengerStructureInductiveReasoning2003a} is for a logic different
  from $\HA$, the transformation proof reasons on the level of cyclic trees and thus readily
  transfer to the setting of $\CHA$.
\end{proof}

\begin{restatable}{lemma}{goodio}
  \label{lem:good-io}
  If $\CHA \vdash \Gamma \sdash \delta$ then there exists an injective
  $\CHA$-proof of $\Gamma \sdash \delta$ with induction order $\sqsubseteq^*$.
\end{restatable}

\begin{proof}
  By \Cref{lem:exists-injective} and \Cref{thm:ind-order-SD}.
\end{proof}

\subsection{Translating Proofs with $\sqsubseteq^*$ induction orders to $\AHA$}
\label{sec:io-to-sha}

In the beginning of \Cref{sec:combinatorics} we hinted at a normal form of
$\CHA$-proofs which is embodied by the proof system $\AHA$. We can now make this
precise: A $\CHA$-proof $\pi$ is in said normal form if it is injective and its
induction order is $\sqsubseteq^*$. Thus, \Cref{lem:good-io} proves that every
$\CHA$-proof can be transformed into one exhibiting this normal form. The
remainder of this section is concerned with translating $\CHA$-proofs in said
normal form to $\AHA$-proofs.

The translation from $\CHA$-proofs to $\AHA$-proofs is essentially verbatim,
i.e. each $\CHA$-rule is translated to its corresponding $\AHA$-rule. The only
difficulty of the translation is assigning each sequent a suitable stack
$\Lambda$ of companion labels and inserting the `$\Lambda$-bookkeeping rules'
$\RDrop$ and $\RComp$ at the right locations.
Towards the assignment of the $\Lambda$, consider a $\CHA$-proof $\pi = ((T,
\beta), \rho)$ with an induction order
$\preceq$ and
recall that \( \gamma(t) \) denotes the local cycle from $[\beta(t), t]$ for $t
\in \dom(\beta)$. For any given node $s \in T$ of $\pi$, the induction order
guarantees that $x_u \in \FV(\lambda(s))$, i.e. certain $x_u$ occurs free in the
sequent at $s$, for certain $u \in \dom(\beta)$. More concretely, we know this
to be the case if $s \in \gamma(t)$ and $u \preceq t$, i.e. $s$ occurs on a
local cycle $\gamma(t)$ which must \emph{preserve} the variable $x_u$ associated
with a local cycle $\gamma(u)$ such that $u \preceq t$.
We consider the set $S(s) \subseteq \dom(\beta)$ which collects these $u$:
\[
	S(s) = \{u \in \dom(\beta) ~|~ \exists t \in \dom(\beta).~s \in \gamma(t) \wedge t \sqsubseteq^* u\}
\]
and begin by proving that, when imposing a suitable ordering on them, the $S(s)$ can indeed be
treated as a stack. In the following, we denote by \( <_\beta \) the order on \(
\dom(\beta) \) induced by \( \beta \), given by $s <_\beta t$ iff $\beta(s) <
\beta(t)$.

\begin{restatable}{lemma}{alphatrans}
  \label{lem:alpha}
  Let $\pi = ((T, \beta), \lambda)$ be an injective $\CHA$-proof with induction order
  $\sqsubseteq^*$. 
  The set \( S(s) \) is linearly ordered by \( <_\beta \) for every \( s \in T \).
  Moreover, the function $\sigma : T \to \dom(\beta)^*$ assigning to \( s \in T
  \) the enumeration of \( S(s) \) according to \(<_\beta\) satisfies:
  \begin{enumerate}
  \item If $\varepsilon \not\in \im(\beta)$ then $\sigma(\varepsilon) =
    \varepsilon$; otherwise $\sigma(\varepsilon) = \beta^{-1}(\varepsilon)$.
  \item If $t \in \Chld(s)$ and \( t \not\in \im(\beta) \), then \( \sigma(t)
    \leq \sigma(s) \).
  \item If $t \in \Chld(s)$ and \( t \in \im(\beta) \), then there is $u \leq \sigma(s)$
    such that $\sigma(t) = u \cdot \beta^{-1}(t)$.
  \end{enumerate}
\end{restatable}

\begin{proof}
  First, notice that per construction, if $t \in S(s)$ then $\beta(t) \leq s$.
  Hence, $S(s)$ is indeed linearly ordered by $<_\beta$ and the
  function $\sigma$ given above well-defined. Furthermore, $\abs{\beta^{-1}(t)}
  \leq 1$ by the injectivity of $\beta$.

  Define $\Down(t) \coloneq \{s \in \dom(\beta) ~|~ s <_\beta t\}$ for $t \in \dom(\beta)$.

  \begin{enumerate}[1.]
  \item We know that $S(\varepsilon) = \{\beta^{-1}(\varepsilon) \mid \text{if }
    \varepsilon \in \im(\beta) \}$ because $\varepsilon \in \gamma(t)$ entails
    $t = \beta^{-1}(\varepsilon)$.
  \item[2. and 3.] 
    We first prove
    that if $t \in \Chld(s)$ then $S(t) \subseteq S(s) \cup \beta^{-1}(t)$: Fix
    some $v \in S(t)$, i.e. such that there is $u$ with $t \in \gamma(u)$ and $u
    \sqsubseteq^* v$. There are two cases:
    \begin{description}
    \item[$\beta(u) \neq t$:] Then $s \in \gamma(u)$, as $t \in \Chld(s)$, and thus $v \in S(s)$.
    \item[$\beta(u) = t$:] Then we consider two cases arising from $u
      \sqsubseteq^* v$:
      \begin{description}
      \item[$u = v$:] Then $u \in \beta^{-1}(t)$.
      \item[$u \sqsubset v_1 \sqsubseteq^* v$:] Then $\beta(v_1) < t$ and thus $s \in \gamma(v_1)$, meaning $v \in S(s)$.
      \end{description}
    \end{description}

    What remains to be shown towards 2.\ and 3.\ is that buds are only removed
    in an $\leq_\beta$-upwards closed manner: if $t \in \Chld(s)$ and $u
    \in S(s) \setminus S(t)$, i.e.\ $u$ was removed from $S(-)$ in stepping from
    $s$ to $t$, then there is no $u' \in S(t)$ such that $u \leq_\beta u' \neq \beta^{-1}(t)$.

    We instead prove an equivalent statement: For $t
    \in \Chld(s)$, if $u \in S(t)$ with $u \neq \beta^{-1}(t)$ then $S(s) \cap
    \Down(u) \subseteq S(t)$. If $u \in S(t)$, there must be $v$ such
    that $t \in \gamma(v)$ and $v \sqsubseteq^* u$. As $\beta^{-1}(u) \neq t$, $s
    \in \gamma(u)$ as well. Now fix $u' \in S(s)
    \cap \Down(u)$, meaning there is $v'$ with $s \in \gamma(v')$ and $v'
    \sqsubseteq^* u'$, as $u' \in S(s)$, and $u' <_\beta u$, as $u' \in \Down(u)$.
    We distinguish two cases:
    \begin{description}
    \item[$v = v'$:] Then $u' \in S(s)$ as $v = v' \sqsubseteq^* u'$.
    \item[$v \neq v'$:] Because of $\beta(v), \beta(v') < s < v, v'$, we may
      distinguish two further cases:
      \begin{description}
      \item[$v \sqsubset v'$:]
        Then $v \sqsubset v' \sqsubseteq^* u'$ and thus $u' \in S(t)$.
      \item[$v' \sqsubset v$:] 
        This case is resolved with a further case distinction:
        For any $x$, if $v <_\beta x$ and $x \sqsubseteq^* u'$ then either $v \sqsubseteq^* u'$ or
        $v <_\beta u'$. From $v' \sqsubset v$ it then follows that $v <_\beta v'$ and
        thus:
        \begin{description}
        \item[$v \sqsubseteq^* u'$:] Then $u' \in S(t)$ as $t \in \gamma(v)$.
        \item[$v <_\beta u'$:] But $u' <_\beta u$ together with $v \sqsubseteq^*
          u$ yield $u' <_\beta u \leq_\beta v$, ruling out this option.
        \end{description}

        It remains to prove the aforementioned disjunction, which we do per induction on the $\sqsubset$-chain
        witnessing $x \sqsubseteq^* u'$: If $x = u'$ then $v <_\beta x = u'$. Thus
        suppose $x \sqsubset x' \sqsubseteq^* u'$. Because $v, x' \leq_\beta x$,
        we may distinguish two cases:
        \begin{description}
        \item[$x' \leq_\beta v$:] Then $x' \leq_\beta v <_\beta x \leq x'$ and thus $v \sqsubset x' \sqsubseteq^* u'$.
        \item[$v <_\beta x'$:] Then continue per inductive hypothesis with $x \coloneq x'$.
        \end{description}
      \end{description}
    \end{description}
  \end{enumerate}
\end{proof}

\begin{remark}
  Intuitively, \Cref{lem:alpha} can be read as stating that the preorder of $\pi$ being
  $\sqsubseteq$ means the $S(s)$ `behave as stacks' when ordered by $<_\beta$
  and the injectivity of $\pi$ means that these `stacks' are extended by at most
  one bud between two nodes, mirroring the $\RComp$-rule.
\end{remark}

Having accounted for the companion stacks $\Lambda$, the rest of the translation
is straight-forward.

\chatoaha*
\begin{proof}
  By \Cref{lem:good-io}, there exists an injective $\CHA$-proof $\pi = ((T, \beta), \rho)$ of
  $\Gamma \sdash \delta$ with $\sqsubseteq^*$ as its induction order. Let
  $\sigma : T \to \dom(\beta)^*$ be the labelling function of $\pi$ as defined
  in \Cref{lem:alpha}. We prove, per induction on the tree structure of
  $T$, starting at the leaves, that for each $s \in T$ there exists a $\AHA$-derivation of
  $\Lambda^s_{\sigma(s)} \mid \lambda(s)$. The companion stacks
  $\Lambda^s_{\sigma(s)}$ are recursively defined as follows:
  \[
    \Lambda^s_\varepsilon \coloneq \varepsilon
    \qquad
    \Lambda^s_{t \cdot u} \coloneq
    \begin{cases}
      (x^+_t \mapsto \lambda(s)); \Lambda^s_u & \text{if } [\beta(t), s] \text{
        passes through the right premise of } \RCase_{x_t} \\
      (x^-_t \mapsto \lambda(s)); \Lambda^s_u & \text{otherwise}
    \end{cases}
  \]
  We argue via a case distinction on the form of $s$:
  \begin{description}
  \item[$s \in \dom(\beta)$:] Then $\sigma(s) = u \cdot s \cdot u'$ for some $u, u' \in
    \dom(\beta)^*$ because $s \in S(s)$ if $s \in \dom(\beta)$. As $\pi$ has an induction order,
    $\RCase_{x_s}$ is applied along $\gamma(s)$. Hence, $\Lambda^s_{\sigma(s)} =
    \Lambda^s_u; (x_s^+ \mapsto \lambda(s)) ; \Lambda^s_{u'}$. The derivation
    below is as desired.
    \begin{comfproof}
      \AXC{}
      \LSC{$\RBud$}
      \UIC{$\Lambda^s_u; x^+_s \mapsto \lambda(s) \mid \lambda(s)$}
      \LSC{$\RDrop$}
      \UIC{$\Lambda^s_u; (x^+_s \mapsto \lambda(s)) ; \Lambda^s_{u'} \mid \lambda(s)$}
    \end{comfproof}
  \item[$\rho(s)$ an axiom:] Then $\varepsilon \mid \lambda(s)$ is an
    axiom of $\AHA$, meaning the derivation below is as desired, where $\dagger$
    is the axiom of $\AHA$ corresponding to $\rho(s)$.
    \begin{comfproof}
      \AXC{}
      \LSC{$\dagger$}
      \UIC{$\varepsilon; \mid \lambda(s)$}
      \LSC{$\RDrop$}
      \UIC{$\Lambda^s_{\sigma(s)}; \mid \lambda(s)$}
    \end{comfproof}
  \item[$s$ an inner node:] 
    Then $s$ was derived from $\Chld(s) = \{t_1, \ldots, t_n\}$ by an
    application of the $\CHA$-rule $R$ and there are $\AHA$-derivations of
    $\Lambda^{t_i}_{\sigma(t_i)} \mid \lambda(t_i)$.

    The rule $R$ has a
    corresponding $\AHA$-rule of the same name. Because of the well-formedness
    restriction on $\AHA$-sequents, the corresponding $\AHA$-rule may only be
    applied to derive $\Lambda^s_{\sigma(s)} \mid \lambda(s)$ if for every companion label
    $(x^\bullet \mapsto \ldots) \in \Lambda^s_{\sigma(s)}$ we have $x \in
    \FV(\lambda(s))$. In other words, if $x_{s'} \in \FV(\lambda(s))$ for every
    $s' \in S(s)$. As argued when we defined $S(s)$, the $s' \in S(s)$ are
    precisely those for which the induction order guarantees that $x_{s'} \in
    \FV(\lambda(s))$ and thus the $\AHA$-rule $R$ may be applied.

    To derive $\Lambda_{\sigma(s)}^s \mid
    \lambda(s)$ from $ \bigl\{ \Lambda_{\sigma(t_i)}^{t_i} \mid \lambda(t_i) \bigm| 1 \le i \le n \bigr\}$ we must
    distinguish two cases for each $t_i$: either $t_i \in \im(\beta)$ or not.
    For the purpose of illustration, let $t_1 \not\in \im(\beta)$ and $t_2 \in
    \im(\beta)$. By \Cref{lem:alpha} $\sigma(t_1) \leq \sigma(s)$ and there
    exists $u \leq \sigma(s)$ such that $\sigma(t_2) = u \cdot s'$ for
    $\beta(s') = t_2$.
    The following derivation, in which $t_3, \ldots, t_n$ are treated
    analogously to $t_1$ or $t_2$ based on the aforementioned case-distinction,
    is as desired:
    \begin{comfproof}
      \AXC{$\Lambda^{t_1}_{\sigma(t_1)} \mid \lambda(t_1)$}
      \LSC{$\RDrop$}
      \UIC{$\Lambda^{t_1}_{\sigma(s)} \mid \lambda(t_1)$}
      \AXC{$\Lambda^{t_2}_{u}; x^-_{s'} \mapsto \lambda(t_2) \mid \lambda(t_2)$}
      \LSC{$\RComp$}
      \UIC{$\Lambda^{t_2}_{u} \mid \lambda(t_2)$}
      \LSC{$\RDrop$}
      \UIC{$\Lambda^{t_2}_{\sigma(s)} \mid \lambda(t_2)$}
      \AXC{$\dotsm$}
      \AXC{$\Lambda^{t_i}_{\sigma(t_i)} \mid \lambda(t_i)$}
      \DOC{}
      \noLine
      \UIC{$\Lambda^{t_i}_{\sigma(s)} \mid \lambda(t_i)$}
      \AXC{$\dotsm$}
      \LSC{$R$}
      \QuinaryInfC{$\Lambda^s_{\sigma(s)} \mid \lambda(s)$}
    \end{comfproof}
    In the case of $t_2$ (and other companions), observe that $\Lambda^{t_2}_{\sigma(t_2)} =
    \Lambda^{t_2}_{u} ; x^-_{s'} \mapsto \lambda(t_2)$. Note also that the
    $\RDrop$-applications account for the fact that $\sigma(t_1) \leq \sigma(s)$
    and $\sigma(t_2) = u \cdot s'$ for $u \leq \sigma(s)$. For the derivation above
    to be valid, it remains to argue that the premises of $R$ indeed have
    companion stacks $\Lambda^{t_i}_{\sigma(s)}$ rather than
    $\Lambda^{s}_{\sigma(s)}$ (note the difference in superscripts). If $R \neq
    \RCase_{x}$ then $\Lambda^s_{u} = \Lambda^{t_i}_{u}$ because if $[\beta(s'),
    t_i]$ passes through the right-hand premise of $\RCase_{x_{s'}}$, this
    must have taken place below $s$ and thus the marker on
    $x^\bullet_{s'}$ in $\Lambda^{s}_{u}$ and $\Lambda^{t_i}_{u}$ for each
    $s'$ occurring in $u$ coincide. If $R = \RCase_{x}$ then this is more
    subtle. The companion labelling of the right-hand premise of $\RCase_{y}$ is
    $(\Lambda^s_{\sigma(s)})^{+x}$. This is precisely
    $\Lambda^{t_2}_{\sigma(s)}$ which, by an argument analogous to that of the
    previous case, differs from $\Lambda^s_{\sigma(s)}$ only by the markings on
    $y$, which are all $y^+$ in $\Lambda^{t_2}_{\sigma(s)}$ whereas
    $\Lambda^{s}_{\sigma(s)}$ may contain instances of $y^-$. On the
    other hand, $\Lambda_{\sigma(s)}^{s} = \Lambda_{\sigma(s)}^{t_1}$ for the
    left-hand of $\RCase_{y}$.
  \end{description}

  From this inductive argument, we can conclude that either $\varepsilon \mid
  \Gamma \sdash \delta$ (if $\varepsilon \not\in \im(\beta)$) or
  $x_{s} \mapsto \lambda(\varepsilon) \mid \Gamma \sdash \delta$ (if $\beta(s)
  = \varepsilon$). From the latter, $\varepsilon \mid \Gamma \sdash \delta$ can
  be derived by an application of $\RComp$. Thus $\AHA \vdash \Gamma \sdash \delta$.
\end{proof}

\section{Conclusion}
\label{sec:conclusion}

Combining the results of \Cref{sec:translation,sec:combinatorics} yields
the desired translation.

\begin{theorem}\label{lem:full}
  If $\CHA \vdash \Gamma \sdash \delta$ then $\HA \vdash \Gamma \sdash \delta$.
\end{theorem}
\begin{proof}
  If $\CHA \vdash \Gamma \sdash \delta$ then $\AHA \vdash \Gamma \sdash \delta$
  by \Cref{lem:cha-to-aha}. Then $\HA \vdash \Gamma \sdash \delta$ by \Cref{lem:key}.
\end{proof}

In this section, we analyse various aspects of \Cref{lem:full} and compare the
method we present to others in the literature. We begin by discussing other
translation methods for cyclic proof systems of first-order arithmetic in
\Cref{sec:related-work}.
In \Cref{sec:pa}, we extend the
translation result to Peano arithmetic. \Cref{sec:log-comp} demonstrates that
our method yields Das' \cite{dasLogicalComplexityCyclic2020} logical complexity
bound $\CP_n \subseteq \IP_{n + 1}$ for $\PA$ and an analogous result for $\HA$.
In \Cref{sec:proof-size} we analyse the proof size increase incurred by our
method.
We close by discussing the applicability of our method
to other cyclic proof systems in \Cref{sec:applicabilty}.

\subsection{Related work}
\label{sec:related-work}

The literature contains translations of cyclic proofs into inductive proofs
for various logics~\cite{das_cut-free_2017,curzi_computational_2023,kuperberg_cyclic_2021,das_cyclic_2023,das_circular_2021,shamkanov_circular_2014,curzi_cyclic_2021,nollet_local_2018}. 
  Restricting attention to Heyting or Peano arithmetic, there
are two distinct approaches to such
translations already present in the literature. The first, introduced by
Simpson~\cite{simpsonCyclicArithmeticEquivalent2017} and later refined by
Das~\cite{dasLogicalComplexityCyclic2020}, proceeds by formalising the soundness
of cyclic Peano arithmetic in $\ACA$, obtaining $\PA$ proofs via the
conservativity of $\ACA$ over said system. The second approach, put forward by
Berardi and
Tatsuta~\cite{berardiEquivalenceInductiveDefinitions2017,berardiEquivalenceIntuitionisticInductive2017},
uses Ramsey-style order-theoretic principles formalisable in $\HA$ to obtain an
induction principle mirroring the cyclic structure of a $\CHA$ or $\CPA$ proof,
using which translating the original proof to $\HA$ or $\PA$, respectively,
becomes straight-forward.

\subsection{Peano arithmetic}
\label{sec:pa}

A proof system $\PA$ of Peano arithmetic can be obtained by modifying $\HA$ in the usual manner:
Considering multi-conclusion sequents $\Gamma \sdash \Delta$ instead of
single-conclusion sequents and replacing the intuitionistic logical rules of
$\HA$ with their classical counterparts. A cyclic proof system $\CPA$ for Peano
arithmetic can be obtained via analogous modifications to $\CHA$, also exchanging
the $\RCase_x$-rule for its multi-conclusion variant. Crucially, the notion of
$\CPA$-proofhood is exactly the same as that of $\CHA$.

\Cref{lem:full} extends to Peano arithmetic in a straightforward manner:
\begin{theorem}\label{lem:full-pa}
  If $\CPA \vdash \Gamma \sdash \Delta$ then $\PA \vdash \Gamma \sdash \Delta$.
\end{theorem}
\begin{proof}
  Let \( \varphi^* \) denote the Gödel--Gentzen translation of \( \varphi \). A
  simple extension of the standard argument shows that $\CPA \vdash \Gamma \sdash \Delta$ entails $\CHA \vdash
  \Gamma^*, \neg \Delta^* \sdash \bot$. By \Cref{lem:full}, one obtains $\HA
  \vdash \Gamma^*, \neg \Delta^* \sdash \bot$ which entails $\PA \vdash \Gamma
  \sdash \Delta$ by the standard argument.
\end{proof}

It should be noted that the method presented in this article can also be applied
to $\CPA$ directly. For this, a system $\APA$ can be derived by modifying $\AHA$
in the usual way. The translation from $\APA$ to $\PA$ proceeds analogously as
that in \Cref{sec:translation}, except that $\widehat{I_\Lambda} := \forall \vec{u} \leq \vec{v}. \forall
  \vec{w}.(\bigwedge \Gamma \to \bigvee \Delta)[\vec{u} / \vec{v}]$ to account
for the multi-conclusion sequents of the classical sequent calculi. The
results of \Cref{sec:combinatorics} also directly transfer to the setting of
Peano arithmetic.

\subsection{Logical complexity}
\label{sec:log-comp}

Using \Cref{lem:full}, one can translate a $\CHA$-proof $\pi$ into a $\HA$-proof
$\pi'$ with the same endsequent. In this section, we characterise the logical
complexity of $\pi'$ in terms of that of $\pi$.

We begin by fixing appropriate notions of logical complexity for $\HA$- and $\CHA$-proofs.
Note that the notion of $\CC$ presented below is based on the notions $\CP_n$
and $\CS_n$ presented in~\cite{dasLogicalComplexityCyclic2020}.

\begin{definition}
  Let $\lcls$ be a set of formulas.

  A $\HA$-proof $\pi$ is said to be in $\IC$
  if every instance of the induction axioms in $\pi$ is of the form
  \( \varphi(0), \forall x. \varphi(x) \to \varphi(Sx) \sdash \varphi(s)  \)
  with $\varphi \in \lcls$. If a sequent $\Gamma \sdash \delta$ is proven by a
  $\HA$-proof $\pi$ in $\IC$ we write $\IC \vdash \Gamma \sdash \delta$.

  A $\CHA$-proof $\pi$ is said to be in $\CC$
  if every bud and companion of $\pi$ is labelled by sequents $\Gamma \sdash
  \delta$ with $\Gamma, \delta \subseteq \lcls$.
  If a sequent $\Gamma \sdash \delta$ is proven by a
  $\CHA$-proof $\pi$ in $\CC$ we write $\CC \vdash \Gamma \sdash \delta$.
\end{definition}

Our characterisation relies on the following two operations on sets of formulas $\lcls$.

\begin{definition}
  Let $\lcls$ be a set of formulas.
  \begin{enumerate}
  \item We define the set $\lcls \sdash \lcls \coloneq \{ (\bigwedge \Phi) \to \psi \mid \Phi
    \subseteq \lcls, \psi \in \lcls\}$.
  \item We define $\Pi(\lcls)$ to be the smallest set of formulas extending \( \Theta \) that is:
    \begin{enumerate}[(i)]
    \item closed under bounded quantification: If $\varphi \in
      \Pi(\lcls)$,
      $x \in \var$ and $t$ is a term not containing \( x \), then $\forall x < t.\varphi \in
      \Pi(\lcls)$ and $\exists x < t.\varphi \in \Pi(\lcls)$, and
    \item closed under unbounded universal quantification: If $\varphi \in \Pi(\lcls)$ and $x
      \in \var$ then $\forall x.\varphi \in \Pi(\lcls)$.
    \end{enumerate}
  \end{enumerate}
\end{definition}

\begin{theorem}\label{lem:ha-complex}
  Let $\lcls$ be a set of formulas. If $\CC \vdash \Gamma \sdash \delta$ then
  $\IF \vdash \Gamma \sdash \Delta$.
\end{theorem}
\begin{proof}
  Let $\pi \vdash \Gamma \sdash \delta$ be a $\CHA$-proof in $\CC$. By
  \Cref{lem:exists-injective} there exists an injective $\CHA$-proof $\pi_i
  \vdash \Gamma \sdash \delta$. By inspecting the proof of
  \Cref{lem:exists-injective}, we observe that $\pi_i$ is still in $\CC$ because it
  has the same sequents as $\pi$. Applying \Cref{thm:ind-order-SD}
  yields a $\CHA$-proof $\pi_{\sqsubseteq^*} \vdash \Gamma \sdash \delta$ in $\CC$ with induction order
  $\sqsubseteq^*$. By
  \Cref{lem:cha-to-aha}, there then exists a $\AHA$-proof $\pi_{\AHA} \vdash \Gamma
  \sdash \delta$. Because all companions and buds of $\pi_{\sqsubseteq^*}$ are labelled with
  sequents $\Gamma' \sdash \delta'$ such that $\Gamma', \delta' \subseteq
  \lcls$, every companion label $x^\bullet \mapsto \Gamma' \sdash \delta'$
  occurring in $\pi_{\AHA}$ is such that $\Gamma', \delta' \subseteq \lcls$.

  It remains to scrutinise the translation of $\AHA$-proofs to $\HA$-proof in
  \Cref{lem:key}. We begin by analysing the logical complexity of the inductive
  hypotheses employed in the translation. Recall that the formula $\widehat{I}_\Lambda$ for
  $\Lambda = \Lambda'; x^\bullet \mapsto \Gamma \sdash \delta$ is defined as $\forall \vec{u} \leq \vec{v}. \forall \vec{w}.~(\bigwedge
  \Gamma \to \delta)[\vec{u} / \vec{v}]$. If $\Gamma, \delta \subseteq
  \lcls$ then $\bigwedge \Gamma \to \delta \in \lcls \sdash \lcls$ and
  $\widehat{I}_\Lambda \in \fcls$.
  Applying \Cref{lem:key} to $\pi_\AHA$ yields a $\HA$-proof $\pi' \vdash \Gamma
  \sdash \delta$. To verify that $\pi'$ is $\IF$, it suffices to check that all
  instances of induction axioms in $\pi'$ are on $\fcls$- or $\Delta_0$-formulas. Note first
  that \Cref{lem:key} inserts various $\Delta_0$-inductions into $\pi'$ to
  derive various properties about the relation $x < y$.
  Except for these, the only
  induction axioms inserted by \Cref{lem:key} occur in the case of the
  $\RComp$-rule and are instances of the following derived rule of $\HA$
  \[
    \inference[$\RInd$]{\Gamma, \forall y < x.~\widehat{I}_\Lambda[y / x] \sdash \widehat{I}_{\Lambda}}{\Gamma \sdash \forall x. \widehat{I}_\Lambda}.
  \]
  To derive the rule above, only one induction on $I' \coloneq \forall x \leq
  y.~\widehat{I}_\Lambda$ is required. Because every companion label of
  $\pi_\AHA$ contains a $\lcls$-sequent, $\widehat{I}_\Lambda \in
  \fcls$ and thus $I' \in \fcls$.
\end{proof}

The notions of $\IC$ and $\CC$ transfer directly to the setting of Peano arithmetic.
In~\cite{dasLogicalComplexityCyclic2020}, Das proves for $\PA$ that $\CS_n = \IS_{n + 1}$,
or equivalently, that $\CP_n = \IP_{n + 1}$. \Cref{lem:ha-complex} can be used
to conclude $\CP_n \subseteq \IP_{n + 1}$, the more involved direction of this equality.

\begin{theorem}
  In $\PA$, if $\CP_n \vdash \Gamma \sdash \Delta$ then $\IP_{n+1} \vdash \Gamma
  \sdash \Delta$.
\end{theorem}
\begin{proof}
  Let $\pi \vdash \Gamma \sdash \Delta$ be a $\CPA$ proof in $\CP_n$. As
  described in \Cref{lem:full-pa}, the usual argument yields a proof $\pi^*
  \vdash \Gamma^*, \neg \Delta^* \sdash \bot$. Crucially, every companion in
  $\pi$ with sequent $\Upsilon \sdash \Xi$ is transformed into a companion in
  $\pi'$ with sequent $\Upsilon^*, \neg \Xi^* \sdash \bot$. Because $\Upsilon,
  \Xi \subseteq \Pi_n$ we obtain $\Upsilon^*, \Xi^* \subseteq \Pi_n$ and $\neg
  \Xi^* \subseteq \Sigma_n$, making $\Upsilon^*, \neg \Xi^* \sdash \bot$ a
  $\Delta_{n + 1}$-sequent. As any companion in $\pi'$ originates from a
  companion in $\pi$, this means $\pi'$ is in $\CD_{n + 1}$. Applying
  \Cref{lem:ha-complex} yields a $\HA$-proof $\pi''' \vdash \Upsilon^*, \neg \Xi^*
  \sdash \delta$ in $\IPar{\Delta_{n + 1}}$ which is extended into a $\PA$-proof $\pi'''' \vdash
  \Upsilon \sdash \Xi$ via the standard argument. It is easily observed that
  $\pi''''$ is still in $\IPar{\Delta_{n + 1}}$. Classically, $\rcls{\Delta_{n +
    1}} \subseteq
  \Pi_{n + 1}$ and thus $\pi''''$ is in $\IP_{n + 1}$ as desired.
\end{proof}

\subsection{Proof size complexity}
\label{sec:proof-size}

Another result of Das~\cite{dasLogicalComplexityCyclic2020} relates $\CPA$ and
$\PA$ in terms of proof size complexity: they show that translating from $\CPA$
to $\PA$ need only incur a singly-exponential size-increase, although this
`proof size efficient' translation maps $\CS_n$ proofs to $\IS_{n + 2}$, rather
than $\IS_{n + 1}$.

Our translation contains two sources of significant blow-up in proof size: (1) a
super-exponential blow-up when transforming arbitrary $\CHA$-proofs into
$\CHA$-proofs of cycle normal form (see \cite[Theorem
6.3.6]{brotherstonSequentCalculusProof2006}) and (2) an exponential blow-up
transforming injective $\CHA$-proofs in cycle normal form to $\CHA$-proofs with
$\sqsubseteq^*$ induction orders via \Cref{lem:good-io}. All other translation
steps only incur polynomial size increase. Thus, we only provide an
improvement on Das' translation witnessing $\CS_n \subseteq \IS_{n + 1}$, which exhibits a
non-uniform size increase.

\subsection{Applicability of the method}
\label{sec:applicabilty}

We conjecture that both the method of Sprenger and Dam, and our refinement
thereof, are more widely applicable
for translating cyclic into inductive proofs, even for logics unrelated to
arithmetic.
As we see it, there are three major requirements on a cyclic proof system
and its logic to be amenable to the translation methods:

\paragraph{Internalising progress} The relations $<$ and
  $\leq$ are key to the formulation of the inductive hypotheses. They are used
  to express which inductive hypothesis is `ready' to be applied and for which a
  progress step must yet occur, both in Sprenger and Dam's method and our
  refinement.
  Transferring the methods to
  other logics requires formulating suitable inductive hypotheses which will
  likely rely on analogous methods of `internalising' the progress of the cyclic
  proof system in the logic.
  At this stage, however, it is unclear how the idea of internalisation can
  be formulated in a general sense that is applicable to systems that lack
  the explicit expressivity to reason about approximations of fixed points,
  for example, type systems and modal logics.
\paragraph{Cut} The treatment of companions, both in Sprenger and
  Dam~\cite{sprengerStructureInductiveReasoning2003a} and in our method, relies on
  the $\RCut$-rule to introduce a new inductive hypothesis. This seems likely to
  be a key feature of the translation method, meaning any system amenable to it
  must feature a $\RCut$-rule or $\RCut$-admissability.
\paragraph{Induction orders} To carry out the method as presented by Sprenger
  and Dam~\cite{sprengerStructureInductiveReasoning2003a} or us, all cyclic proofs
  of the system must be justifiable by induction orders, or it must at least be
  possible to transform them in such a way that they always become justifiable by
  induction orders. For trace conditions in
  which the objects traced along branches (e.g. variables, terms, formulas) always
  have a unique successor, as in $\CHA$, an argument analogous to
  \Cref{lem:exists-io} can likely be carried out. It is known that not all cyclic
  proof systems exhibit this property, for example those that allow
  `contractions' in their traces such as
    Simpson's~\cite{simpsonCyclicArithmeticEquivalent2017} cyclic proof system
    for PA. A possible path to extending our method to
  such cases is provided by the work of Afshari and
  Leigh~\cite{afshariFinitaryProofSystems2016}: They provide a translation from
  cyclic to inductive proofs for the modal $\mu$-calculus. Rather than relying on
  induction orders, their translation is obtained via reset proof systems, another
  soundness condition for cyclic proof system. Similarly to Sprenger and Dam's
  approach, the method can be split into two parts: (1) bringing arbitrary cyclic
  proofs into a combinatorial normal-form and (2) translating proofs in said
  normal-form into inductive proofs. While part (2) of their method was found to
  contain a mistake~(see \cite{kloibhofer_note_2023}), part (1) remains correct. Indeed, their notion of
  normal-form bears a close resemblance to $\sqsubseteq^*$ induction orders.
  As Leigh and
  Wehr~\cite{leighGTCResetGenerating2023} demonstrates, most cyclic proof
  systems admit equivalent reset systems, so adapting the method presented here to
  the setting of reset-proofs should yield a method that is more widely applicable.

\printbibliography{}

\end{document}